\documentclass[11pt]{amsart}
\usepackage{a4wide} 
\usepackage{graphicx}
\usepackage[colorlinks=true, urlcolor=blue, linkcolor=blue, citecolor=green]{hyperref} 
\usepackage{amsmath,amssymb,amsfonts,enumerate}
\usepackage{amsthm}
\usepackage{esint}
\usepackage{mathrsfs}
\usepackage{mathtools}
\usepackage{mathscinet}
\usepackage[nameinlink]{cleveref}
\usepackage[dvipsnames]{xcolor} 

\theoremstyle{plain}
\newtheorem{theorem}{Theorem}[section]
\newtheorem{lemma}[theorem]{Lemma}
\newtheorem{proposition}[theorem]{Proposition}
\newtheorem{corollary}[theorem]{Corollary}

\newtheorem{thmletter}{Theorem}

\theoremstyle{definition}

\theoremstyle{remark}

\newtheorem{conjecture}[theorem]{Conjecture}

\numberwithin{equation}{section}

\title[]{The $C^{p'}$-regularity conjecture near $p=2$}

\author{Se-Chan Lee}
\address{School of Mathematics, Korea Institute for Advanced Study, Seoul 02455, Republic of Korea}
\email{sechan@kias.re.kr}

\author{Taehun Lee}
\address{Department of Mathematics, Konkuk University, 120 Neungdong-ro, Gwangjin-gu, Seoul 05029, Republic of Korea}
\email{taehun@konkuk.ac.kr}



\thanks{Se-Chan Lee is supported by the KIAS Individual Grant (No. MG099001) at the Korea Institute for Advanced Study.
Taehun Lee is supported by a National Research Foundation of Korea (NRF) grant funded by the Korean government (MSIT) (Grant No. RS-2023-00211258).}

\begin{document} 

\begin{abstract}
    We prove the $C^{p'}$-regularity conjecture in every dimension when $p>2$ is sufficiently close to $2$. To this end, we establish improved H\"older estimates for the gradients of $p$-harmonic functions. These estimates also determine the first-order asymptotics of the optimal gradient H\"older exponent in every dimension. The proof combines compactness, harmonic rigidity of the limiting profiles, and a sharp uniform gap estimate for the first variation of the gradient excess.
\end{abstract}

\maketitle

\section{Introduction}
In this paper, we are concerned with the optimal regularity of weak solutions $u \in W^{1, p}(B_1)$ of the $p$-Poisson equation
\begin{equation}\label{eq-plaplace}
    -\Delta_pu=f\in L^{\infty}(B_1)  \quad \text{in $B_1 \subset \mathbb{R}^d$},
\end{equation}
where $p \in (2, \infty)$ and $d \geq 2$. In the linear case $p=2$, weak solutions of $-\Delta u=f \in L^{\infty}(B_1)$ are locally of class $C^{1, \alpha}$ for every $\alpha \in (0,1)$, but may fail to belong to $C^{1,1}$; see \cite[Section~2.2]{FRRO22} for instance. When $p>2$, by contrast, the $p$-Laplacian is degenerate at critical points of a solution, and the regularity theory becomes considerably more delicate. Nevertheless, the classical regularity theory guarantees that weak solutions of \eqref{eq-plaplace} are locally of class $C^{1, \beta}$ for some $\beta=\beta(d, p) \in (0,1)$. The local H\"older continuity of the gradient for $p$-harmonic functions (i.e., continuous weak solutions of \eqref{eq-plaplace} with $f \equiv 0$), and more generally for wide classes of degenerate quasilinear elliptic equations, was established in the classical works of Ural'ceva~\cite{Ura68}, Uhlenbeck~\cite{Uhl77}, Evans~\cite{Eva82}, DiBenedetto~\cite{DB83}, Lewis~\cite{Lew83}, and Tolksdorf~\cite{Tol84}. Related boundary and global regularity estimates were developed by Lieberman~\cite{Lie88}. For further developments in the regularity theory of $p$-Laplacian type problems, we refer to Duzaar--Mingione~\cite{DM04}, Kuusi--Mingione~\cite{KM12},  Cianchi--Maz'ya~\cite{CM14}, and the references therein.

It is then natural to ask what the optimal H\"older exponent $\beta$ for the gradient is. To illustrate the issue, let $p' \coloneqq p/(p-1)$ be the H\"older conjugate of $p$. It follows from a direct calculation that
\begin{equation*}
    \Delta_p(|x|^{p'})=d(p')^{p-1}.
\end{equation*}
This simple observation shows that solutions of \eqref{eq-plaplace} cannot in general have better regularity than 
\begin{equation*}
    C^{p'}=C^{1, \frac{1}{p-1}},
\end{equation*}
and gives rise to the following well-known open problem.
\begin{conjecture}[$C^{p'}$-regularity conjecture]
    Solutions of \eqref{eq-plaplace} are locally of class $C^{p'}$.
\end{conjecture}

Ara\'ujo--Teixeira--Urbano proved this conjecture in the plane \cite{ATU17} and established it for several restricted classes of solutions in higher dimensions \cite{ATU18}. Their approach provides a general reduction scheme: the conjectured $C^{1,1/(p-1)}$ regularity follows once $p$-harmonic functions satisfy uniform $C^{1,\alpha}$ estimates for some $\alpha>1/(p-1)$; see Theorem~\ref{thm-ATU18}.

In the plane, Iwaniec--Manfredi~\cite{IM89} used complex-analytic and quasiregular mapping techniques to prove that every $p$-harmonic function is locally of class $C^{1, \alpha^{\ast}(2, p)}$, where
\begin{equation*}
    \alpha^{\ast}(2, p)\coloneqq\frac{1}{6} \left(\frac{p}{p-1}+\sqrt{1+\frac{14}{p-1}+\frac{1}{(p-1)^2}} \right),
\end{equation*}
and this exponent is optimal. However, the proof in \cite{IM89} does not provide the quantitative control of the $C^{1, \alpha^{\ast}}$ norm. Ara\'ujo--Teixeira--Urbano~\cite{ATU17} obtained uniform $C^{1, \alpha_{\mathrm{ATU}}(2, p)}$ estimates for $p$-harmonic functions, where
\begin{equation*}
    \alpha_{\mathrm{ATU}}(2, p) \coloneqq \frac{1}{2p}\left(-3-\frac{1}{p-1}+\sqrt{33+\frac{30}{p-1}+\frac{1}{(p-1)^2}} \right).
\end{equation*}
Since $\alpha^{\ast}(2, p)>\alpha_{\mathrm{ATU}}(2, p)>1/(p-1)$, these estimates allow us to utilize the reduction described above. See also da Silva~\cite{dS26} and Lindgren--Lindqvist~\cite{LL17} for related results.

In higher dimensions $d \geq 3$, no counterpart of the planar complex-analytic theory is available, and only partial results are known. The results in \cite{ATU18}, for example, apply to several classes of solutions such as radially symmetric solutions and those without saddle critical points. In the perturbative regime for \eqref{eq-plaplace} near $p=2$, Pimentel--Rampasso--Santos~\cite{PRS20} obtained local $C^{1,\alpha}$ estimates, for each fixed $\alpha \in (0,1)$, provided that $0<p-2<\varepsilon=\varepsilon(d,\alpha)$. Nevertheless, the conjecture has remained open for general solutions in higher dimensions.\\

Our first main theorem provides a positive answer to the conjecture when $p>2$ is sufficiently close to $2$.

\begin{theorem}[$C^{p'}$-regularity conjecture]\label{cor-conjecture}
    Let $d \geq 2$. Then there exists $\varepsilon_d>0$ such that  if $p \in (2, 2+\varepsilon_d)$ and $u \in W^{1,p}(B_1)$ is a weak solution of \eqref{eq-plaplace}, then $u \in C^{1, 1/(p-1)}_{\mathrm{loc}}(B_1)$ with the uniform estimate
    \begin{equation*}
        \|u\|_{C^{1, \frac{1}{p-1}}(\overline B_{1/2})} \leq C\left(\|u\|_{L^p(B_1)}+\|f\|_{L^{\infty}(B_1)}^{\frac{1}{p-1}} \right),
    \end{equation*}
    where $C>0$ depends only on $d$ and $p$.
\end{theorem}

The proof of Theorem~\ref{cor-conjecture} is based on the compactness method developed in \cite{ATU17, ATU18}; related compactness and improvement of flatness arguments were introduced by Imbert--Silvestre~\cite{IS13} and Ara\'ujo--Ricarte--Teixeira~\cite{ART15}. We establish the required quantitative improvement of $p$-harmonic regularity in the following theorem, which is our second main theorem.

\begin{theorem}[Improved regularity of $p$-harmonic functions]\label{thm-main-p-harmonic}
    Let $d \geq 2$ and fix $c \in (0, 1/2)$. Then there exists $\varepsilon_{d, c}>0$ such that  if $p \in (2, 2+\varepsilon_{d, c})$ and  $u \in W^{1,p}(B_1)$ is $p$-harmonic in $B_1$, then 
    \begin{equation*}
        u \in C^{1, \beta_{p,c}}_{\mathrm{loc}}(B_1)\quad \text{for} \quad \beta_{p,c} \coloneqq \frac{1}{p-1}+c(p-2) \in (0,1).
    \end{equation*}
    Moreover, we have the uniform estimate
    \begin{equation*}
        \|u\|_{C^{1, \beta_{p,c}}(\overline B_{1/2})} \leq C\|u\|_{L^{\infty}(B_1)},
    \end{equation*}
    where $C>0$ depends only on $d$, $p$, and $c$.
\end{theorem}

Let us discuss the meaning and sharpness of the upper bound $1/2$ for $c$ in Theorem~\ref{thm-main-p-harmonic}. A planar $p$-harmonic function may be regarded as a $p$-harmonic function in $\mathbb{R}^d$ that is independent of the remaining variables. Writing $p=2+\varepsilon$, the optimal exponent $\alpha^{\ast}(2, p)$ from \cite{IM89} satisfies
 \begin{equation}\label{eq-optimal-IM}
     \alpha^{\ast}(2, 2+\varepsilon)=1-\frac{1}{2}\varepsilon+O(\varepsilon^2),
 \end{equation}
 while the quantitative exponent $\alpha_{\mathrm{ATU}}(2, p)$ from \cite{ATU17} satisfies
 \begin{equation*}
     \alpha_{\mathrm{ATU}}(2, 2+\varepsilon)=1-\frac{3}{4}\varepsilon+O(\varepsilon^2).
 \end{equation*}
Although the estimates in \cite{ATU17} were sufficient to deduce the $C^{p'}$-regularity conjecture, where
\begin{equation*}
    p'-1=1-\varepsilon+O(\varepsilon^2),
\end{equation*}
they did not recover the optimal first-order coefficient $-1/2$ of $\alpha^{\ast}(2,p)$ for $p$-harmonic functions. 

On the other hand, we observe that
 \begin{equation*}
     \beta_{p,c}=\frac{1}{1+\varepsilon}+c\varepsilon=1-(1-c)\varepsilon+O(\varepsilon^2).
 \end{equation*}
By comparing this expansion with \eqref{eq-optimal-IM}, the range $c<1/2$ in Theorem~\ref{thm-main-p-harmonic} is asymptotically optimal at the level of the first-order term. In other words, Theorem~\ref{thm-main-p-harmonic} implies the following corollary, which determines the first-order behavior of the optimal gradient H\"older exponent for $p$-harmonic functions as $p\to2^+$. For $d\geq2$ and $p>2$, let $\alpha_{\mathrm{opt}}(d,p)$ denote the supremum of all $\alpha\in(0,1)$ such that every $p$-harmonic function in dimension $d$ is locally of class $C^{1,\alpha}$; in particular, $\alpha_{\mathrm{opt}}(2, p)=\alpha^{\ast}(2, p)$.

\begin{corollary}
\label{cor-optimal-asymptotics}
For every $d\geq2$,
\begin{equation}\label{eq-optimal-asymptotics}
    \lim_{p\to 2^+}
    \frac{
        \alpha_{\mathrm{opt}}(d,p)-\frac{1}{p-1}
    }{p-2}
    =
    \frac12.
\end{equation}
\end{corollary}

To the best of our knowledge, this result was only known for $d=2$. For $d \geq 3$, the lower bound in \eqref{eq-optimal-asymptotics} follows from Theorem~\ref{thm-main-p-harmonic} by letting $c \to 1/2^-$, while the upper bound follows from the sharp planar examples of \cite{IM89} and their trivial extensions to higher dimensions.\\

Let us illustrate the main ideas in the proof of Theorem~\ref{thm-main-p-harmonic}. Away from the critical set $\{Du=0\}$, the $p$-Laplace equation is uniformly elliptic and the standard Schauder theory gives stronger regularity of $u$. Thus, motivated by the Campanato characterization, which relates H\"older gradient estimates to uniform decay of the gradient excess 
\begin{equation*}
         \mathcal{E}_r(u)= \left( \fint_{B_r} |Du-(Du)_{B_r}|^2\right)^{1/2},
\end{equation*}
we investigate the decay of this excess at critical points. For harmonic functions, this excess enjoys a sharp linear decay estimate, with equality only for quadratic profiles. Since the desired exponent $\beta_{p,c}$ converges to $1$ as $p \to 2^+$, qualitative compactness alone does not suffice to complete the contradiction argument. To overcome this difficulty, we further present an asymptotic analysis of the first variation of $\mathcal{E}_r(u)$. The argument proceeds through the following steps.

\medskip
\noindent\textit{Step 1: Compactness.} We suppose that the desired excess decay fails at a critical point. After subtracting an affine function and normalizing by the gradient excess, we obtain a normalized solution $v_{\varepsilon}$ for $\varepsilon=p-2$, and let $h_{\varepsilon}$ be its harmonic replacement. Then Lemma~\ref{lem-compactness} shows that both $v_{\varepsilon}$ and $h_{\varepsilon}$ converge to a harmonic function $h$ in appropriate senses. Moreover, the first-order correction $w_{\varepsilon}=(v_{\varepsilon}-h_{\varepsilon})/{\varepsilon}$ converges to a function $w$ satisfying
\begin{equation*}
    -\Delta w=\mathrm{div}(Dh \log|Dh|) \quad \text{in $B_1$}.
\end{equation*}
We note that this equation is closely related to the first-order term in the expansion
\begin{equation*}
    |z|^{\varepsilon}z=z+\varepsilon z\log|z|+o(\varepsilon) \quad \text{as $\varepsilon \to 0$}.
\end{equation*}

\medskip
\noindent\textit{Step 2: Harmonic rigidity.} The next step is a rigidity argument that characterizes the limiting harmonic function $h$. In the normalized setting, the gradient excess of $h$ satisfies 
\begin{equation*}
    \mathcal{E}_{\rho}(h) \leq \rho \mathcal{E}_1(h) \quad \text{for every $\rho \in (0,1)$}.
\end{equation*}
If the inequality is strict, then compactness already gives the desired improvement of regularity. In the equality case, a spherical harmonic decomposition shows that 
\begin{equation*}
    h=P_H(x) \coloneqq x^{\top}Hx
\end{equation*}
for a nonzero trace-free symmetric matrix $H$; see Lemma~\ref{lem-harmonic-rigidity} for details. In short, it only remains to find a contradiction when $h$ is a homogeneous harmonic polynomial of degree two.

\medskip
\noindent\textit{Step 3: Uniform gap.} When $h=P_H$, the limit equation becomes
\begin{equation*}
    -\Delta w_H=\mathrm{div}(DP_H \log|DP_H|) \quad \text{in $B_1$}.
\end{equation*}
The first variation of the gradient excess is given by the bilinear form
\begin{equation*}
    \mathcal{B}_{\rho}(P_H, w_H)=\mu(H) \rho^2\log\rho(\mathcal{E}_1(P_H))^2,
\end{equation*}
where $\mu(H)$ denotes the coefficient of the degree-two component of $w_H$ in the $P_H$-direction; see Lemma~\ref{lem-identity}. Therefore, the problem essentially reduces to estimating the quantity $\mu(H)$ for every nonzero trace-free symmetric matrix $H$ in a uniform sense.

We establish in Proposition~\ref{prop-uniform-gap}, by a direct deterministic method, that 
\begin{equation*}
    \inf \mu(H) \geq -1+\delta_d \quad \text{for some $\delta_d>0$},
\end{equation*}
where the infimum is taken over all nonzero trace-free symmetric matrices $H$. This uniform gap yields some H\"older gradient exponent strictly larger than $1/(p-1)$ for $p$-harmonic functions. Therefore, it is sufficient to prove the $C^{p'}$-regularity conjecture in Theorem~\ref{cor-conjecture} with the aid of \cite[Theorem~2]{ATU18}.

Let us explain why $-1$ is the natural threshold for the quantity $\mu(H)$ in a heuristic sense. Suppose formally that normalized solutions $v_{\varepsilon}$ admit an asymptotic expansion
\begin{equation*}
    v_{\varepsilon}=P_H+\varepsilon w_H+o(\varepsilon),
\end{equation*}
where $\mathcal{E}_1(P_H)=1$. Then Lemmas~\ref{lem-harmonic-rigidity} and \ref{lem-identity} imply, for each $\rho \in (0,1)$, that
\begin{equation*}
    (\mathcal{E}_{\rho}(v_{\varepsilon}))^2 = (\mathcal{E}_{\rho}(P_H))^2+2\varepsilon \mathcal{B}_{\rho}(P_H, w_H)+o(\varepsilon)=\rho^2(1+2\varepsilon\mu(H)\log\rho)+o(\varepsilon).
\end{equation*}
By using the Taylor expansion $\rho^{2+2\gamma}=\rho^2(1+2\gamma \log \rho)+O(\gamma^2)$, we obtain
\begin{equation*}
    \mathcal{E}_{\rho}(v_{\varepsilon}) =\rho^{1+\varepsilon \mu(H)}+o(\varepsilon).
\end{equation*}
Since the desired $C^{p'}$-scale corresponds to 
\begin{equation*}
    \rho^{\frac{1}{p-1}}=\rho^{1-\varepsilon}+O(\varepsilon^2),
\end{equation*}
$-1$ is the relevant threshold for $\mu(H)$. This observation will be rigorously justified in Lemma~\ref{lem-initial-step}. To obtain the full range $c\in(0,1/2)$ in Theorem~\ref{thm-main-p-harmonic}, we further establish the sharp bound $\mu(H)\geq-1/2$ in Appendix~\ref{sec-appendix}; see Proposition~\ref{prop:mu-lb}. 

\medskip
\noindent\textit{Step 4: Iteration.} By combining all steps above, one can obtain an improved excess decay at critical points. After distinguishing two regimes, namely, near-critical regime (Lemma~\ref{lem-near-critical}) and non-degenerate regime (Lemma~\ref{lem-non-degenerate}), a rather standard iteration argument allows us to finish the proof of Theorem~\ref{thm-main-p-harmonic}.\\

 We next present two applications of Theorem~\ref{cor-conjecture}. First, it yields the optimal spatial regularity for the parabolic $p$-Laplace equation near $p=2$. To be precise, let $u$ be a bounded weak solution of
\begin{equation*}
    u_t=\Delta_pu \quad\text{in $Q_1$},
\end{equation*}
where $ Q_r=B_r\times(-r^2,0]$, and suppose that
$p\in(2,2+\varepsilon_d)$.
Since the time derivative $u_t$ is locally bounded by Lee--Lian--Yun--Zhang~\cite[Theorem~1.1]{LLYZ25}, an application of Theorem~\ref{cor-conjecture} to $-\Delta_pu=-u_t$ on each time slice gives interior $C^{1,1/(p-1)}$ estimates in the spatial variables. Combining these estimates with the standard interpolation argument leads to
\begin{equation*}
    |Du(x,t)-Du(y,s)|
    \leq C\left(
        |x-y|^{1/(p-1)}+|t-s|^{1/p}
    \right)
\end{equation*}
for all $(x,t),(y,s)\in  Q_{1/2}$, where $C>0$ depends only on $d$, $p$, and $\|u\|_{L^\infty( Q_1)}$.

The second application concerns the optimal interior $C^{p'}$ regularity for the obstacle problem for the $p$-Laplacian. Let $u$ be a minimizer of
\begin{equation*}
    \int_{B_1}\left(\frac1p|Dv|^p-fv\right)\,\mathrm{d}x
\end{equation*}
among functions above an obstacle $\varphi\in C^2(\overline B_1)$ with smooth admissible Dirichlet data, where $f\in L^\infty(B_1)$. The Lewy--Stampacchia inequality (see \cite[Section~2, (2.5)]{Rod05} for instance) gives
\begin{equation*}
    f\leq-\Delta_pu\leq\max\{f,-\Delta_p\varphi\}
    \quad\text{a.e. in $B_1$}.
\end{equation*}
Since $p>2$ and $\varphi\in C^2(\overline B_1)$, we have $\Delta_p\varphi \in L^{\infty}(B_1)$, and so Theorem~\ref{cor-conjecture} guarantees that $u\in C^{p'}_{\mathrm{loc}}(B_1)$ for $p\in(2,2+\varepsilon_d)$. It is noteworthy that Andersson--Lindgren--Shahgholian~\cite[Theorem~1]{ALS15} developed the corresponding pointwise $C^{p'}$-growth estimate of $u$ at free boundary points $x_0 \in \partial\{u>\varphi\}$, while Theorem~\ref{cor-conjecture} provides the optimal $C^{p'}$ regularity throughout the interior, both away from and across the free boundary.

\bigskip

 Let us finally describe the organization of the paper. In Section~\ref{sec-preliminaries}, we introduce the gradient excess and the normalization used throughout the paper. In Section~\ref{sec-compactness}, we prove the compactness of the normalized solutions, their harmonic replacements, and the first-order corrections. In Sections~\ref{sec-harmonic-rigidity} and \ref{sec-uniform-gap}, we establish the harmonic rigidity result and investigate the first variation around quadratic profiles to identify the coefficient $\mu(H)$, respectively. Section~\ref{sec-proof-main} contains the excess decay and iteration arguments, and the proofs of Theorem~\ref{thm-main-p-harmonic} and Theorem~\ref{cor-conjecture}. Finally, Appendix~\ref{sec-appendix} is devoted to the proof of the sharp lower bound for $\mu(H)$.

\section{Preliminaries}\label{sec-preliminaries}
Let $B_1=B_1(0) \subset \mathbb{R}^d$ for $d \geq 2$ and $p \in (2, \infty)$. For $f\in H^1(B_r)$ with $r \in (0,1]$, we define the \emph{gradient excess} (or the \emph{excess}, for short)
\begin{equation*}
         \mathcal{E}_r(f) \coloneqq \left( \fint_{B_r} |Df-(Df)_{B_r}|^2\,\mathrm{d}x\right)^{1/2},
\end{equation*}
where
\begin{equation*}
    (Df)_{B_r} \coloneqq \fint_{B_r}Df(x)\,\mathrm{d}x  \in \mathbb{R}^d.
\end{equation*}
For $f, g \in H^1(B_r)$ with $r \in (0 ,1]$, we also define the associated symmetric bilinear form
\begin{equation*}
    \mathcal{B}_{r}(f, g) \coloneqq\fint_{B_{r}}(Df-(Df)_{B_{r}}) \cdot (Dg-(Dg)_{B_{r}})\,\mathrm{d}x.
\end{equation*}
Since
\begin{equation*}
    \mathcal{B}_r(f, g)=\frac{1}{2}\left.\frac{\mathrm{d}}{\mathrm{d}t}\right|_{t=0}(\mathcal{E}_r(f+tg))^2,
\end{equation*}
$\mathcal{B}_r(f, g)$ can be understood as the first variation at $f$ in the direction $g$ of the (squared) gradient excess.\\

Since we are interested in the near $p=2$ regime, we let $\varepsilon=p-2>0$ which will be sufficiently small. Suppose that $u$ is a $p$-harmonic function in $B_1$ such that $Du(0)=0$ and $\mathcal{E}_r(u)>0$ with $r \in (0,1]$. We define the \emph{normalized solution} $v$ by
\begin{equation*}
    v(x)=v_{\varepsilon}(x) \coloneqq \frac{u(rx)-u(0)-(Du)_{B_r} \cdot (rx)}{r\mathcal{E}_r(u)}
\end{equation*}
and the \emph{normalized slope} $\xi$ by
\begin{equation*}
    \xi=\xi_{\varepsilon}\coloneqq\frac{(Du)_{B_r}}{\mathcal{E}_r(u)}.
\end{equation*}
 Here the subscript $\varepsilon$ records the dependence on $p=2+\varepsilon$; we may omit it if there is no confusion. Then it is immediate to check that
\begin{equation}\label{eq-normalized1}
    -\mathrm{div}(|\xi+Dv|^{p-2}(\xi+Dv))=0 \quad \text{in $B_1$}
\end{equation}
 and
\begin{equation}\label{eq-normalized2}
    v(0)=0, \quad (Dv)_{B_1}=0, \quad \fint_{B_1}|Dv|^2=1, \quad \text{and} \quad \xi+Dv(0)=0.
\end{equation}
Moreover, any pair $(v, \xi) \in W^{1, p}(B_1) \times \mathbb{R}^d$ satisfying \eqref{eq-normalized1} and \eqref{eq-normalized2} will be called a \emph{normalized pair}.

The \emph{harmonic replacement} of $v=v_{\varepsilon} \in W^{1, p}(B_1)$ is the unique function $h=h_{\varepsilon} \in H^1(B_1)$ satisfying
\begin{equation}\label{eq-harmonic-replacement}
    h-v \in H_0^1(B_1) \quad \text{and} \quad \int_{B_1} Dh \cdot D\varphi=0 \quad \text{for any $\varphi \in H_0^1(B_1)$}.
\end{equation}
It is easy to check that $(Dh)_{B_1}=(Dv)_{B_1}=0$. We finally define 
\begin{equation*}
    w_{\varepsilon} \coloneqq \frac{v_{\varepsilon}-h_{\varepsilon}}{\varepsilon},
\end{equation*}
which represents the \emph{first-order correction} to the harmonic replacement $h_{\varepsilon}$ in the asymptotic expansion of $v_{\varepsilon}$ as $\varepsilon \to 0$. In fact, the Taylor expansion in the parameter $\varepsilon$ gives that
    \begin{equation*}
        |z|^{\varepsilon}z=z+\varepsilon z\log|z|+o(\varepsilon) \quad \text{as $\varepsilon \to 0$},
    \end{equation*}
and so formally,
\begin{equation*}
    -\Delta_pf=-\mathrm{div}(|Df|^{\varepsilon}Df)\approx-\Delta f-\varepsilon\mathrm{div}(Df \log|Df|).
\end{equation*}

We end this section with a consequence of \cite[Theorem~2]{ATU18} that connects Theorems~\ref{cor-conjecture} and \ref{thm-main-p-harmonic}.

\begin{thmletter}[{\cite[Theorem~2]{ATU18}}]
\label{thm-ATU18}
Let $d\geq2$ and $p>2$.
Assume that there exist
$\alpha_0\in(1/(p-1),1)$ and $C_0>0$ such that every
bounded $p$-harmonic function $h\in W^{1,p}(B_1)$ satisfies
\begin{equation*}
    \|h\|_{C^{1,\alpha_0}(\overline B_{1/2})}
    \leq C_0\|h\|_{L^\infty(B_1)}.
\end{equation*}
Then every weak solution $u\in W^{1,p}(B_1)$ of
\begin{equation*}
    -\Delta_pu=f\in L^\infty(B_1) \quad\text{in $B_1$}
\end{equation*}
belongs to
$C^{1,1/(p-1)}_{\mathrm{loc}}(B_1)$ and satisfies
\begin{equation*}
    \|u\|_{C^{1,\frac{1}{p-1}}(\overline B_{1/2})}
    \leq C\left(
        \|u\|_{L^p(B_1)}
        +\|f\|_{L^\infty(B_1)}^{\frac{1}{p-1}}
    \right),
\end{equation*}
where $C>0$ depends only on $d$, $p$, $\alpha_0$, and $C_0$.
\end{thmletter}

\section{Compactness}\label{sec-compactness}
\begin{lemma}\label{lem-xi-bound}
    There exist $\varepsilon_0\in (0,1)$, $\gamma_0 \in (0,1)$, and $C_d>0$, depending only on $d$, with the following property: if $p=2+\varepsilon \in (2, 2+\varepsilon_0)$ and $(v, \xi)$ satisfies \eqref{eq-normalized1} and \eqref{eq-normalized2}, then
    \begin{equation*}
        |\xi| \leq C_d \quad \text{and} \quad \|v\|_{C^{1, \gamma_0}(\overline B_{2/3})} \leq C_d.
    \end{equation*}
\end{lemma}

\begin{proof}
    We set
    \begin{equation*}
         \tilde v(x) \coloneqq \frac{v(x)+\xi\cdot x}{(1+|\xi|^2)^{1/2}}.
    \end{equation*}
    It is immediate to check that $\tilde v \in W^{1, p}(B_1)$ is $p$-harmonic with $D\tilde v(0)=0$ and
    \begin{equation*}
        \fint_{B_1}|D\tilde v|^2=\frac{\fint_{B_1}|Dv+\xi|^2}{1+|\xi|^2}=1.
    \end{equation*}

    \begin{itemize}
        \item  We first fix $\varepsilon_0$ small enough so that $p_0 \coloneqq 2+\varepsilon_0 <2d/(d-2)=2^{\ast}$ for $d \geq 3$. Then the Sobolev--Poincar\'e inequality shows that
    \begin{equation*}
        \|\tilde v-(\tilde v)_{B_1}\|_{L^{p_0}(B_1)} \leq C_d\|D\tilde v\|_{L^2(B_1)} \leq C_d.
    \end{equation*}
        For $d=2$, we use the embedding $H^1(B_1) \hookrightarrow L^{q}(B_1)$ for any fixed $q \in (2, \infty)$.
    
      \item The positive and negative parts of $\tilde v-(\tilde v)_{B_1}$ are nonnegative $p$-subharmonic functions. Applying \cite[Lemma~2.8]{DP06} to each part, with $\mu=(p_0-1)^{-1}$, gives
    \begin{equation*}
        \|\tilde v-(\tilde v)_{B_1}\|_{L^{\infty}(B_{7/8})} \leq C_d\|\tilde v-(\tilde v)_{B_1}\|_{L^{p_0}(B_1)}.
    \end{equation*}

   \item By applying \cite[Lemma~2.5]{DP06} for $\tilde v-(\tilde v)_{B_1}$ with $\mu=(p_0-1)^{-1}$, we obtain $\gamma_0=\gamma_0(\mu,d)\in(0,1)$,
    independent of $p\in(2,2+\varepsilon_0)$, such that the uniform estimate
    \begin{equation*}
        \|\tilde v-(\tilde v)_{B_1}\|_{C^{1, \gamma_0}(\overline B_{2/3})} \leq C_d \|\tilde v-(\tilde v)_{B_1}\|_{L^{\infty}(B_{7/8})}
    \end{equation*}
   is available.
    \end{itemize}
    A combination of these estimates yields that
    \begin{equation}\label{eq-Z}
        \|\tilde v-(\tilde v)_{B_1}\|_{C^{1, \gamma_0}(\overline B_{2/3})} \leq C_d.
    \end{equation}
    In particular, we have
    \begin{equation*}
        |Dv(x)+\xi| \leq C_d|x|^{\gamma_0}(1+|\xi|^2)^{1/2} \quad \text{for $x\in B_{2/3}$}.
    \end{equation*}
    If $|\xi| \leq 1$, then there is nothing to prove; we may assume that $|\xi| \geq 1$. Then choose $\tau=\tau_d \in (0,2/3)$ small enough so that
    \begin{equation*}
        |Dv(x)+\xi| \leq |\xi|/2 \quad \text{in $B_{\tau}$},
    \end{equation*}
    which implies
     \begin{equation*}
        1=\fint_{B_1}|Dv|^2 \geq \frac{|B_{\tau}|}{|B_1|} \frac{|\xi|^2}{4}.
    \end{equation*}
    This gives the desired bound on $|\xi|$. Since $\tilde v(0)=0$, \eqref{eq-Z} gives $|(\tilde v)_{B_1}|\leq C_d$. The estimate for $v$ now follows from \eqref{eq-Z} together with these bounds.
\end{proof}

We now prove an important compactness lemma regarding first-order corrections $w_{\varepsilon}$.
\begin{lemma}[Compactness]\label{lem-compactness}
     Let $\varepsilon_j \to 0$, and let $(v_j, \xi_j)$ satisfy \eqref{eq-normalized1} and \eqref{eq-normalized2} with $p_j=2+\varepsilon_j$. Let $h_j$ be the harmonic replacement of $v_j$, and set $w_j=(v_j-h_j)/\varepsilon_j$. Then up to a subsequence, the following hold:
     \begin{enumerate}[(i)]
         \item $\xi_j \to 0$ in $\mathbb{R}^d$;

         \item $h_j \rightharpoonup h$ weakly in $H^1(B_1)$ and $h_j \to h$ locally uniformly in $B_1$ for some harmonic $h$ with $h(0)=Dh(0)=0$;

         \item $v_j \to h$ strongly in $C^1(\overline B_R)$ for every $R \in (0,1)$;

         \item $w_j \rightharpoonup w$ weakly in $W^{1,q}(B_1)$ for some $q\in (1, 2)$ and in $H^1(B_R)$ for every $R \in (0,1)$, where $w \in W_0^{1,q}(B_1)$ is the unique solution of 
    \begin{equation*}
        -\Delta w=\mathrm{div}(Dh\log|Dh|) \quad \text{in $B_1$}.
    \end{equation*}
     \end{enumerate}
\end{lemma}

\begin{proof}
    Let us first develop the estimates for $w_j$. For any normalized pair $(v, \xi)$, let $h$ be its harmonic replacement and set $w_{\varepsilon}=(v-h)/\varepsilon \in H_0^1(B_1)$. Since $h$ is harmonic and $v$ satisfies
    \begin{equation*}
        -\mathrm{div}(|\xi+Dv|^{p-2}(\xi+Dv))=0,
    \end{equation*}
    we have
    \begin{equation*}
        -\Delta w_{\varepsilon}=\mathrm{div}G_{\varepsilon}(\xi+Dv),
    \end{equation*}
    where
    \begin{equation*}
        G_{\varepsilon}(z) \coloneqq z\frac{|z|^{\varepsilon}-1}{\varepsilon} \quad \text{for $z \in \mathbb{R}^d$}.
    \end{equation*}
     Here we may assume that $\varepsilon \in (0,\varepsilon_0)$ for $\varepsilon_0>0$ chosen in Lemma~\ref{lem-xi-bound}. Then we observe that
    \begin{itemize}
        \item for $0<t \leq 1$, 
        \begin{equation*}
            t\frac{1-t^{\varepsilon}}{\varepsilon} \leq t|\log t| \leq e^{-1},
        \end{equation*}
        while for $t \geq 1$,
         \begin{equation*}
            t\frac{t^{\varepsilon}-1}{\varepsilon} \leq t^{1+\varepsilon}\log t \leq C_{\eta}t^{1+\varepsilon+\eta} \quad \text{for each $\eta>0$};
        \end{equation*}

        \item it follows from \eqref{eq-normalized2} and Lemma~\ref{lem-xi-bound} that
    \begin{equation*}        \fint_{B_1}|\xi+Dv|^2=|\xi|^2+\fint_{B_1}|Dv|^2=|\xi|^2+1 \leq C_d^2+1,
    \end{equation*}
    that is, $\xi+Dv$ is uniformly bounded in $L^2(B_1)$;

    \item it follows from Lemma~\ref{lem-xi-bound} that 
    \begin{equation*}
        \|v\|_{C^{1,\gamma_0}(\overline B_R)} \leq C_{d, R} \quad \text{for every $R \in (0,1)$}.
    \end{equation*}
    \end{itemize}
    
    We now fix $q \in (1, 2)$ and then choose $\varepsilon_0$ and $\eta$ small enough so that
    \begin{equation*}
        q(1+\varepsilon_0+\eta) \leq 2,
    \end{equation*}
    and so for every $\varepsilon \in (0, \varepsilon_0)$,
    \begin{equation*}
        \|G_{\varepsilon}(\xi+Dv)\|_{L^q(B_1)} \leq  C_d \quad \text{and} \quad \|G_{\varepsilon}(\xi+Dv)\|_{L^{\infty}(B_R)} \leq  C_{d, R}. 
    \end{equation*}
    Therefore, the standard regularity theory for Laplace equations gives that
    \begin{equation}\label{eq-compactness-w}
        \|w_{\varepsilon}\|_{W^{1,q}(B_1)} \leq C_d \quad \text{and} \quad \|w_{\varepsilon}\|_{H^1(B_R)} \leq C_{d,R}.
    \end{equation}

We are now ready to prove the convergence results. It follows from the normalization \eqref{eq-normalized2}, Lemma~\ref{lem-xi-bound}, and the definition of harmonic replacements that $\{v_j\}$ and $\{h_j\}$ are uniformly bounded in $H^1(B_1)$. Thus, up to a subsequence, we have $h_j \rightharpoonup h$ weakly in $H^1(B_1)$ and $h_j \to h$ locally uniformly in $B_1$, where $h$ is harmonic in $B_1$.

By the uniform estimate \eqref{eq-compactness-w}, we have $w_j \rightharpoonup w \in W_0^{1,q}(B_1)$ weakly in $W^{1, q}(B_1)$ and in $H^1(B_R)$ for every $R \in (0,1)$. Moreover, it follows from
\begin{equation*}
    \|v_j-h_j\|_{H^1(B_R)}=\varepsilon_j\|w_j\|_{H^1(B_R)} \to 0
\end{equation*}
and the uniform $C^{1, \gamma_0}$ estimates for $v_j$ that $v_j \to h$ in $C^1(\overline B_R)$ for every $R \in (0,1)$. In particular, \eqref{eq-normalized2} shows that $h(0)=0$.

We next show that $\xi_j$ converges to zero. Since $h_j-v_j \in H_0^1(B_1)$ and $(Dv_j)_{B_1}=0$, we have
\begin{equation*}
   (Dh_j)_{B_1}=(Dv_j)_{B_1}=0.
\end{equation*}
Passing to the limit using the weak convergence in $H^1(B_1)$ and applying the mean value property for $h$ give 
\begin{equation*}
    0=(Dh)_{B_1}=Dh(0).
\end{equation*}
By recalling \eqref{eq-normalized2} and the local $C^1$ convergence, we arrive at
\begin{equation*}
    \xi_j=-Dv_j(0) \to -Dh(0)=0.
\end{equation*}
On the other hand, the function $G_{\varepsilon}$ converges to $z\log|z|$ locally uniformly in $\mathbb{R}^d$. In particular, we have
\begin{equation*}
  G_{\varepsilon_j}(\xi_j+Dv_j) \to Dh\log|Dh| \quad \text{locally uniformly in $B_1$}.
\end{equation*}
By noticing that $w_j$ satisfies the integral identity 
\begin{equation*}
    \int_{B_1} Dw_j \cdot D\varphi=-\int_{B_1} G_{\varepsilon_j}(\xi_j+Dv_j) \cdot D\varphi \quad \text{for any $\varphi \in C_c^{\infty}(B_1)$}, 
\end{equation*}
we can pass to the limit to conclude that
\begin{equation*}
    -\Delta w=\mathrm{div} (Dh\log|Dh|) \quad \text{in $B_1$}. \qedhere
\end{equation*}
\end{proof}

\section{Harmonic rigidity}\label{sec-harmonic-rigidity}
We say that a polynomial $\phi : \mathbb{R}^d \to \mathbb{R}$ is a \emph{homogeneous harmonic polynomial of degree $k$} if 
\begin{equation*}
    \Delta \phi=0 \quad \text{and} \quad \phi(tx)=t^k\phi(x)
\end{equation*}
for every $t \in \mathbb{R}$ and $x \in \mathbb{R}^d$. When $k=2$, every homogeneous polynomial can be written uniquely as $P_A$ for some $A \in \mathrm{Sym}_d\coloneqq \{A \in \mathbb{R}^{d \times d} : A^{\top}=A\}$, where
\begin{equation*}
    P_A(x) \coloneqq x^{\top} A x \quad \text{for $x \in \mathbb{R}^d$}.
\end{equation*}
Moreover, $P_A$ is harmonic if and only if $\operatorname{tr}A=0$. We use the same notation
$P_A$ for its restriction to $\mathbb{S}^{d-1}$, i.e.,
\begin{equation*}
    P_A(\omega)=\omega^{\top} A\omega
    \quad \text{for $\omega \in \mathbb{S}^{d-1}$}.
\end{equation*}

\begin{lemma}\label{lem-harmonic-rigidity}
    Let $h \in H^1(B_1)$ be harmonic in $B_1$ satisfying $(Dh)_{B_1}=0$. Then we have the inequality
    \begin{equation*}
        \mathcal{E}_{\rho}(h) \leq \rho \mathcal{E}_1(h) \quad \text{for each $\rho \in (0,1)$},
    \end{equation*}
    where equality holds if and only if 
    \begin{equation*}
        h-h(0)=P_H(x) \quad \text{for $H \in \mathrm{Sym}_d$ and $\mathrm{tr}H=0$}.
    \end{equation*}
\end{lemma}

\begin{proof}
    It follows from the analyticity of harmonic functions that 
    \begin{equation*}
        h=\sum_{k=0}^{\infty}h_k, \quad \text{where $h_k$ is a polynomial of degree $k$}.
    \end{equation*}
    Since $\Delta h=0$ and the summands $\Delta h_k$ have distinct homogeneities, we find that each $h_k$ is a homogeneous harmonic polynomial of degree $k$. Moreover, since $(Dh)_{B_1}=Dh(0)$ by the mean value property and $h_1(x)=Dh(0)\cdot x$, the condition $(Dh)_{B_1}=0$ implies that $h_1 \equiv 0$. Thus, we have a decomposition
    \begin{equation*}
        h=h(0)+\sum_{k \geq 2} h_k,
    \end{equation*}
    where $h_k$ is a homogeneous harmonic polynomial of degree $k$.

    We next observe that for $k \neq l$,
    \begin{equation*}
        \int_{B_{\rho}}Dh_k \cdot Dh_l=\int_{\partial B_{\rho}}h_k \partial_{n}h_l=\frac{l}{\rho}\int_{\partial B_{\rho}}h_kh_l.
    \end{equation*}
    By interchanging $k$ and $l$, it turns out that
    \begin{equation*}
        \int_{B_{\rho}}Dh_k \cdot Dh_l=0.
    \end{equation*}
    Moreover, since $h_k$ is homogeneous of degree $k$, we have
    \begin{equation*}
        Dh_k(\rho x)=\rho^{k-1}Dh_k(x),
    \end{equation*}
    and so
    \begin{equation*}
        \fint_{B_{\rho}}|Dh_k|^2=\rho^{2(k-1)}\fint_{B_1}|Dh_k|^2.
    \end{equation*}
    Applying the mean value property to each component of $Dh$ yields
    \begin{align*}
        (Dh)_{B_\rho}=Dh(0)=(Dh)_{B_1}=0 \quad \text{for every $\rho \in (0,1)$}.
    \end{align*}
    Thus, we obtain
    \begin{equation*}
        \begin{aligned}
            (\mathcal{E}_{\rho}(h))^2&=\fint_{B_{\rho}}|Dh-(Dh)_{B_{\rho}}|^2=\fint_{B_{\rho}}|Dh|^2=\sum_{k \geq 2}\fint_{B_{\rho}}|Dh_k|^2\\
            &=\sum_{k \geq 2}\rho^{2(k-1)} \fint_{B_1}|Dh_k|^2 \quad \text{for each $\rho \in (0,1)$}.
        \end{aligned}
    \end{equation*}
    Letting $\rho\to 1^-$ in this identity and using the monotone convergence theorem, we have
\begin{equation*}
    (\mathcal E_1(h))^2=\sum_{k\geq2}\fint_{B_1}|Dh_k|^2.
\end{equation*}
    Furthermore, noting that $\rho^{2(k-1)} \leq \rho^2$ for $\rho \in (0,1)$ and $k \geq 2$, we arrive at
    \begin{equation*}
        (\mathcal{E}_{\rho}(h))^2 \leq \rho^{2}\sum_{k \geq 2} \fint_{B_1}|Dh_k|^2=\rho^2 (\mathcal{E}_1(h))^2,
    \end{equation*}
    where the equality holds if and only if $h_k \equiv 0$ for every $k \geq 3$.
\end{proof}

\section{Uniform gap}\label{sec-uniform-gap}
For $H \in \mathrm{Sym}_d \setminus \{0\}$ with $\mathrm{tr}H=0$, define
\begin{equation*}
    f_H(x) \coloneqq 
    \begin{cases}
        -2\dfrac{P_{H^3}(x)}{P_{H^2}(x)} & \text{if $x \in \mathbb{R}^d \setminus \mathrm{ker}H$},\\
        0 & \text{if $x \in \mathrm{ker}H$}.
    \end{cases}
\end{equation*}
 We note that $\mathrm{ker}H$ is of measure zero and $f_H \in L^{\infty}(\mathbb{R}^d)$ due to the inequality $|x^{\top}H^3x| \leq \|H\| x^{\top}H^2x$.
 Since $P_{H^3}$ and $P_{H^2}$ are both homogeneous of degree two,
$f_H$ is homogeneous of degree zero. We use the same notation for
its restriction to $\mathbb{S}^{d-1}$.

\begin{lemma}\label{lem-distribution}
Let $H \in \mathrm{Sym}_d \setminus \{0\}$ with $\mathrm{tr}H=0$. Then
\begin{equation*}
    f_H=-\operatorname{div}(DP_H\log|DP_H|)\quad\text{in }B_1
\end{equation*}
in the sense of distributions. In particular, if $q\in(1,2)$ and $w_H\in W_0^{1,q}(B_1)$ is the solution of
\begin{equation*}
    -\Delta w_H=\mathrm{div}(DP_H\log|DP_H|) \quad \text{in $B_1$},
\end{equation*}
then $\Delta w_H=f_H$ in $B_1$ in the sense of distributions.
\end{lemma}

\begin{proof}
    For $\delta>0$, we let
    \begin{equation*}
        F_{H, \delta} \coloneqq \frac{1}{2}DP_H \log(|DP_H|^2+\delta^2),
    \end{equation*}
    which converges to $DP_H \log|DP_H|$ when $\delta \to 0$. Since $DP_H=2Hx$ and $\mathrm{tr}H=0$, we have for $x \notin \mathrm{ker}H$,
    \begin{equation*}
        \mathrm{div} F_{H, \delta}=\frac{8x^{\top}H^3x}{4x^{\top}H^2x+\delta^2} \to 2\frac{P_{H^3}}{P_{H^2}}=-f_H \quad \text{when $\delta \to 0$}.
    \end{equation*}
    Thus, for any $\varphi \in C_c^{\infty}(B_1)$, 
    \begin{equation*}
        -\int_{B_1}F_{H, \delta} \cdot D\varphi\,\mathrm{d}x=\int_{B_1}\mathrm{div}F_{H, \delta} \varphi\,\mathrm{d}x=\int_{B_1}\frac{8x^{\top}H^3x}{4x^{\top}H^2x+\delta^2} \varphi\,\mathrm{d}x.
    \end{equation*}
    By letting $\delta \to 0$, the dominated convergence theorem yields that
    \begin{equation*}
        -\int_{B_1} DP_H\log|DP_H| \cdot D\varphi\,\mathrm{d}x=-\int_{B_1}f_H \varphi\,\mathrm{d}x
    \end{equation*}
    as desired.
\end{proof}

We point out that Lemma~\ref{lem-distribution} yields higher regularity for $w_H \in W_0^{1,q}(B_1)$. In fact, due to the global $W^{2,s}$ estimates for the Dirichlet problem with bounded right-hand side, we have 
\begin{equation*}
    w_H \in W^{2, s}(B_1) \cap W_0^{1,s}(B_1) \quad \text{for every $s \in (1, \infty)$}.
\end{equation*}
In particular, $\mathcal{E}_1(w_H)$ and $\mathcal{B}_1(P_H, w_H)$ are well defined.\\

We briefly recall the spherical harmonic decomposition. For an integer $k \geq 0$, the \emph{space of spherical harmonics of degree $k$} is defined by
\begin{equation*}
     \mathcal{H}_k=\mathcal{H}_k(\mathbb{S}^{d-1}) \coloneqq \{Y=\phi|_{\mathbb{S}^{d-1}} : \text{$\phi$ is a homogeneous harmonic polynomial of degree $k$}\}.
\end{equation*}
Equivalently, $\mathcal{H}_k(\mathbb{S}^{d-1})$ is the eigenspace consisting of eigenfunctions $Y$ satisfying
\begin{equation*}
    -\Delta_{\mathbb{S}^{d-1}}Y=k(k+d-2)Y,
\end{equation*}
where $\Delta_{\mathbb{S}^{d-1}}$ is the Laplace--Beltrami operator. The eigenspaces are mutually orthogonal in $L^2(\mathbb{S}^{d-1})$ and give the Hilbert decomposition
\begin{equation*}
    L^2(\mathbb{S}^{d-1})=\bigoplus_{k=0}^{\infty} \mathcal{H}_k(\mathbb{S}^{d-1}),
\end{equation*}
equipped with the normalized inner product
\begin{equation*}
    \langle f, g \rangle \coloneqq \fint_{\mathbb{S}^{d-1}}fg.
\end{equation*}
If $\{Y_{k, m}\}_{m=1}^{N_k}$ is an orthonormal basis of $\mathcal{H}_k$, then the \emph{orthogonal projection} is given by
\begin{equation*}
    \Pi_{k}f \coloneqq \sum_{m=1}^{N_k} \langle f, Y_{k, m} \rangle \, Y_{k, m} \quad \text{for $N_k\coloneqq\mathrm{dim}\mathcal{H}_k$}.
\end{equation*}
For functions defined in $B_1$, we understand $\Pi_k$ to act on the angular variable $\omega$ for each fixed radius $r$.\\

We now restrict our attention to the case $k=2$, where we have the concrete characterization
\begin{equation*}
    \mathcal{H}_2(\mathbb{S}^{d-1})=\{P_H : \text{$H \in \mathrm{Sym}_d$ and $\mathrm{tr}H=0$} \}.
\end{equation*}
By writing $x=r\omega$ with $r=|x|>0$ and $\omega \in \mathbb{S}^{d-1}$, we observe that
\begin{equation}\label{eq-spherical-harmonic}
    \Delta (r^mY(\omega))=[m(m+d-2)-2d]r^{m-2}Y(\omega) \quad \text{for $Y \in \mathcal{H}_2$ and $m \in \mathbb{R}$}.
\end{equation}
By letting $m=2$, we obtain that $r^2Y$ is harmonic. Moreover, a differentiation of \eqref{eq-spherical-harmonic} with respect to $m$ at $m=2$ gives the identity
\begin{equation}\label{eq-log-identity}
    \Delta (r^2\log rY(\omega))=(d+2) Y(\omega).
\end{equation}
Finally, we define $\mu(H)\in\mathbb{R}$ by decomposing $\Pi_2f_H \in \mathcal{H}_2$ into the direction of $P_H$ and its orthogonal complement as
\begin{equation}\label{decomposition-F}
    \Pi_2f_H=(d+2)(\mu(H)P_H+\psi_H)
\end{equation}
where $\psi_H\in\mathcal H_2$ and $\langle\psi_H,P_H\rangle=0$. Then the coefficient $\mu(H)$ is a key quantity that describes the first variation of the gradient excess as follows.

\begin{lemma}\label{lem-identity}
     Let $H \in \mathrm{Sym}_d \setminus \{0\}$ with $\mathrm{tr}H=0$. Suppose that $w_H \in H_0^{1}(B_1)$ is the solution of $\Delta w_H=f_H$ in $B_1$. Then we have
    \begin{equation}\label{eq-identity-Q}
        \mathcal{B}_{\rho}(P_H, w_H)=\mu(H)\rho^2\log\rho (\mathcal{E}_1(P_H))^2 \quad \text{for every $\rho \in (0,1)$.}
    \end{equation}
\end{lemma}

\begin{proof}
    We first claim that $w_H$ satisfies
     \begin{equation*}
        \Pi_2w_H=(\mu(H)P_H(x)+\Psi_H(x))\log|x|,
    \end{equation*}
    where  
\begin{equation*}
    \Psi_H (x)\coloneqq |x|^2\psi_H\left(\frac{x}{|x|}\right)
\end{equation*}
is a homogeneous harmonic polynomial of degree two, and is orthogonal to $P_H$. Indeed, for smooth functions $\varphi$, it is immediate to check that
\begin{equation*}
    \Delta(\Pi_2\varphi)=\Pi_2(\Delta\varphi).
\end{equation*}
By standard approximation, this identity remains valid in the sense of
distributions:
\begin{equation*}
    \Delta(\Pi_2w_H)=\Pi_2(\Delta w_H)=\Pi_2f_H=(d+2)(\mu(H)P_H+\psi_H).
\end{equation*}
On the other hand, for
\begin{equation*}
    \tilde w \coloneqq (\mu(H)P_H(x)+\Psi_H(x))\log|x|,
\end{equation*}
we use the identity \eqref{eq-log-identity} to obtain
\begin{equation*}
    \Delta \tilde w=(d+2)(\mu(H)P_H+\psi_H).
\end{equation*}
Since $\Pi_2w_H, \tilde w \in H_0^1(B_1)$ and $\Delta (\Pi_2w_H-\tilde w)=0$, we prove the desired claim.\\

We now prove \eqref{eq-identity-Q}. It is easy to check that $(DP_H)_{B_{\rho}}=0$ and $(Dw_H)_{B_{\rho}}=0$ from the symmetry. In particular, we have
    \begin{equation*}
        \mathcal{B}_{\rho}(P_H, w_H)=\fint_{B_{\rho}}DP_H \cdot Dw_H=\mu(H) \fint_{B_{\rho}}DP_H \cdot D(P_H\log|x|),
    \end{equation*}
    where we used the orthogonality of spherical harmonics with different degrees, $\langle P_H, \psi_H \rangle=0$, and the decomposition of $\Pi_2w_H$. Moreover, the homogeneity gives that for $x=\rho y$,
    \begin{equation*}
        \begin{aligned}
            DP_H(\rho y)&=\rho DP_H(y),\\
            D(P_H\log|x|)(\rho y)&=\rho D(P_H\log|y|)(y)+\rho \log \rho DP_H(y).
        \end{aligned}
    \end{equation*}
    Thus, a direct calculation shows that
    \begin{equation*}
        \begin{aligned}
              \mu(H)\fint_{B_{\rho}}DP_H \cdot D(P_H\log|x|)&=\frac{\mu(H)}{|B_1|} \int_{B_1} DP_H(\rho y) \cdot D(P_H\log|x|)(\rho y)\,\mathrm{d}y\\
              &=\mu(H)\rho^2 \fint_{B_1} DP_H \cdot D(P_H\log |y|)+\mu(H)\rho^2\log\rho \fint_{B_1}|DP_H|^2\\
              &=\rho^2 \mathcal{B}_1(P_H, w_H)+\mu(H)\rho^2\log\rho (\mathcal{E}_1(P_H))^2.
        \end{aligned}
    \end{equation*}
    Since 
    \begin{equation*}
        \mathcal{B}_1(P_H, w_H)=\fint_{B_1} DP_H \cdot Dw_H=-\fint_{B_1} \Delta P_H w_H=0,
    \end{equation*}
    the identity \eqref{eq-identity-Q} is verified.
\end{proof}

\begin{lemma}\label{lem-mu}
    Let $H \in \mathrm{Sym}_d \setminus \{0\}$ satisfy $\mathrm{tr}H=0$, and let $\mu(H)$ be defined by \eqref{decomposition-F}. Then we have
    \begin{equation*}
             \mu(H)=-\frac{d}{\mathrm{tr}(H^2)}\fint_{\mathbb{S}^{d-1}}\frac{P_HP_{H^3}}{P_{H^2}}=-\frac{2}{\mathrm{tr}(H^2)}\fint_{\mathbb{S}^{d-1}}\frac{P_{H^4}P_{H^2}-P^2_{H^3}}{P^2_{H^2}}.
    \end{equation*}
\end{lemma}

\begin{proof}
    We take the inner product of \eqref{decomposition-F} with $P_H$ to find that
\begin{equation*}
    \mu(H)=\frac{\langle f_H, P_H \rangle}{(d+2)\langle P_H, P_H \rangle}=-\frac{2}{(d+2)\langle P_H, P_H \rangle}\fint_{\mathbb{S}^{d-1}} \frac{P_HP_{H^3}}{P_{H^2}}.
\end{equation*}
Then we utilize the well-known algebraic identity on $\mathbb{S}^{d-1}$:
\begin{equation*}
\fint_{\mathbb S^{d-1}}
(\omega^{\top}A\omega)(\omega^{\top}B\omega)\,\mathrm{d}\omega=\frac{\mathrm{tr}A\,\mathrm{tr}B+2\mathrm{tr}(AB)}{d(d+2)} \quad \text{for $A,B \in \mathrm{Sym}_d$}.
\end{equation*}
 By putting $A=B=H$ and using $\mathrm{tr}H=0$, we have
\begin{equation*}
    \langle P_H, P_H \rangle=\frac{2\mathrm{tr}(H^2)}{d(d+2)},
\end{equation*}
which gives the first equality.

For the second equality, we set
\begin{equation*}
    X(\omega) \coloneqq H\omega-P_H(\omega)\omega \quad \text{for $\omega \in \mathbb{S}^{d-1}$}.
\end{equation*}
It is immediate to check that
\begin{itemize}
    \item $X(\omega) \cdot \omega=0$, i.e., $X$ is tangent to $\mathbb{S}^{d-1}$;

    \item $\nabla_{\mathbb{R}^d}P_H=2H\omega$ and so $\nabla_{\mathbb{S}^{d-1}}P_H=(I-\omega \otimes \omega)\nabla_{\mathbb{R}^d}P_H=2X$;

    \item since $P_H \in \mathcal{H}_2(\mathbb{S}^{d-1})$, we have
    \begin{equation*}
        \mathrm{div}_{\mathbb{S}^{d-1}}X=\frac{1}{2}\Delta_{\mathbb{S}^{d-1}}P_H=-dP_H;
    \end{equation*}

    \item $X \cdot \nabla_{\mathbb{S}^{d-1}}P_{H^2}=2(P_{H^3}-P_HP_{H^2})$ and $X \cdot \nabla_{\mathbb{S}^{d-1}}P_{H^3}=2(P_{H^4}-P_HP_{H^3})$.
\end{itemize}
Therefore, we integrate by parts to obtain that
\begin{equation*}
    \begin{aligned}
        -d\fint_{\mathbb{S}^{d-1}} \frac{P_HP_{H^3}}{P_{H^2}}&=\fint_{\mathbb{S}^{d-1}} \frac{P_{H^3}}{P_{H^2}} \mathrm{div}_{\mathbb{S}^{d-1}}X=-\fint_{\mathbb{S}^{d-1}} \nabla_{\mathbb{S}^{d-1}}\left(\frac{P_{H^3}}{P_{H^2}}  \right) \cdot X\\
        &=-2\fint_{\mathbb{S}^{d-1}}\frac{P_{H^4}P_{H^2}-P^2_{H^3}}{P^2_{H^2}}.
    \end{aligned}
\end{equation*}
To justify the calculation near $\mathrm{ker}H$, we may repeat the regularization argument as in Lemma~\ref{lem-distribution}; we omit the details.
\end{proof}

We finish this section by establishing a \emph{uniform gap} above $-1$. Here uniformity means that the size of the gap between $-1$ and $\mu(H)$ is independent of the nonzero trace-free matrix $H$. This is important since the quadratic profile $P_H$ arising from the compactness argument is not fixed in advance. We also note that Proposition~\ref{prop-uniform-gap} below suffices to prove Theorem~\ref{cor-conjecture} with the aid of \cite[Theorem~2]{ATU18}. To prove Theorem~\ref{thm-main-p-harmonic} for the full range $c \in (0, 1/2)$, however, we need the sharp lower bound $\mu(H) \geq -1/2$. The proof of this sharper estimate is more algebraic, and we postpone it to Appendix~\ref{sec-appendix}; see Proposition~\ref{prop:mu-lb}.

\begin{proposition}[Uniform gap]\label{prop-uniform-gap}
    There exists a constant $\delta_d>0$, depending only on $d$, such that 
    \begin{equation*}
         \mu(H) \geq -1+\delta_d
    \end{equation*}
    for every $H \in \mathrm{Sym}_d \setminus \{0\}$ satisfying $\mathrm{tr}H=0$.
\end{proposition}

\begin{proof}
    Since $H$ is symmetric and $H \mapsto \mu(H)$ is invariant under orthogonal change of coordinates, we may assume that 
    \begin{equation*}
        H=\mathrm{diag}(\lambda_1, \ldots, \lambda_d).
    \end{equation*}
    For every $\omega \in \mathbb{S}^{d-1}$ such that $P_{H^2}(\omega)>0$, define
    \begin{equation*}
        \nu_i(\omega) \coloneqq \frac{\lambda_i^2\omega_i^2}{P_{H^2}(\omega)}=\frac{\lambda_i^2\omega_i^2}{\sum_{j=1}^d\lambda_j^2\omega_j^2}.
    \end{equation*}
    Then it is easy to check that $\nu_i \geq 0$, $\sum_i\nu_i=1$, and
    \begin{equation*}
        \frac{P_{H^3}}{P_{H^2}}=\sum_{i=1}^d\nu_i\lambda_i, \quad \frac{P_{H^4}}{P_{H^2}}=\sum_{i=1}^d\nu_i\lambda_i^2.
    \end{equation*}
    Thus, we may write
    \begin{equation*}
        \frac{P_{H^4}P_{H^2}-P^2_{H^3}}{P^2_{H^2}}=\sum_{i=1}^d\nu_i\lambda_i^2-\left(\sum_{i=1}^d\nu_i\lambda_i\right)^2 \eqqcolon V_{\nu}(\lambda),
    \end{equation*}
    and so, in view of Lemma~\ref{lem-mu},
    \begin{equation*}
        \mu(H)=-\frac{2}{\mathrm{tr}(H^2)} \fint_{\mathbb{S}^{d-1}}V_{\nu}(\lambda). 
    \end{equation*}

    We now use this formula for $\mu(H)$ to show the desired uniform gap. To begin with, let 
    \begin{equation*}
        a \coloneqq \max_i \lambda_i, \quad  b \coloneqq \min_i\lambda_i, \quad \text{and} \quad m \coloneqq \sum_{i=1}^d\nu_i\lambda_i.
    \end{equation*}
    Note that $a>0>b$ from $H \neq 0$ and $\mathrm{tr}H=0$. Since $(a-\lambda_i)(\lambda_i-b) \geq 0$, we have
    \begin{equation*}
        \sum_{i=1}^d\nu_i\lambda_i^2 \leq (a+b) \sum_{i=1}^d \nu_i\lambda_i-ab \sum_{i=1}^d\nu_i=(a+b)m-ab.
    \end{equation*}
    Therefore, we conclude that
    \begin{equation*}
        V_{\nu}(\lambda) \leq (a-m)(m-b) \leq \frac{(a-b)^2}{4} \leq  \frac{a^2+b^2}{2} \leq \frac{\mathrm{tr}(H^2)}{2},
    \end{equation*}
   which implies $\mu(H) \geq -1$. In fact, the equality cannot occur; suppose that $\mu(H)=-1$. Then it follows from the equality condition that $a+b=0$ and $a^2+b^2=\mathrm{tr}(H^2)$. In particular, after reordering the coordinates, we have $H=\mathrm{diag}(a, -a, \ldots, 0)$. Since
   \begin{equation*}
       V_{\nu}(\lambda)=a^2(\nu_1+\nu_2)-a^2(\nu_1-\nu_2)^2=4a^2\frac{\omega_1^2\omega_2^2}{(\omega_1^2+\omega_2^2)^2}
   \end{equation*}
    and
    \begin{equation*}
        \frac{\mathrm{tr}(H^2)}{2}=a^2,
    \end{equation*}
    the equality condition again requires that $\omega_1^2=\omega_2^2$ for almost every $\omega \in \mathbb{S}^{d-1}$, which is impossible.

   It only remains to verify the uniform gap between $\mu(H)$ and $-1$. The formula in Lemma~\ref{lem-mu} shows that $\mu(tH)=\mu(H)$ for every $t \neq 0$. Thus, we may assume that $\mathrm{tr}(H^2)=1$ and consider the compact set 
   \begin{equation*}
       K_d \coloneqq \{H \in \mathrm{Sym}_d : \text{$\mathrm{tr}H=0$ and $\mathrm{tr}(H^2)=1$}\}.
   \end{equation*}
    Since $\mu$ is continuous on $K_d$ and $\mu(H)>-1$ for every $H \in K_d$, we obtain
    \begin{equation*}
        \delta_d \coloneqq \min_{H \in K_d} (1+\mu(H))>0,
    \end{equation*}
    which finishes the proof.
\end{proof}

\section{Proof of the main theorems}\label{sec-proof-main}
For $p=2+\varepsilon$ and 
\begin{equation*}
0<c<\hat c \coloneqq \frac{1}{2}\left(c+\frac{1}{2}\right)<\frac{1}{2},    
\end{equation*}
we observe that
\begin{equation*}
    \begin{aligned}
           \beta_{p,c}&=\frac{1}{p-1}+c(p-2)=\frac{1}{1+\varepsilon}+c\varepsilon=1-(1-c)\varepsilon+O(\varepsilon^2),\\
           \beta_{p, \hat c}&=\frac{1}{p-1}+\hat c(p-2)=\frac{1}{1+\varepsilon}+\hat c\varepsilon=1-(1-\hat c)\varepsilon+O(\varepsilon^2).
    \end{aligned}
\end{equation*}
By decreasing the upper bound for $\varepsilon>0$ if necessary, we have
\begin{equation*}
    \frac{1}{p-1} <\beta_{p,c}<\beta_{p, \hat c}<1.
\end{equation*}

\begin{lemma}[Excess decay at critical points]\label{lem-initial-step}
    For every fixed $\rho \in (0, 1/4)$ and $c \in (0, 1/2)$, there exists $\varepsilon_{d, \rho, c}>0$ such that every $(v, \xi)$ satisfying \eqref{eq-normalized1} and \eqref{eq-normalized2}, with $p \in (2, 2+\varepsilon_{d, \rho, c})$, satisfies
    \begin{equation*}
        \mathcal{E}_{\rho}(v) \leq \rho^{\beta_{p, \hat c}}.
    \end{equation*}
\end{lemma}

\begin{proof}
    Fix $\rho \in (0, 1/4)$ and $c \in (0, 1/2)$. Suppose that the conclusion is false for every choice of $\varepsilon_{d, \rho,c}>0$. Then there exist $p_j=2+\varepsilon_j \to 2$ and normalized pairs $(v_j, \xi_j)$ satisfying \eqref{eq-normalized1} and \eqref{eq-normalized2}, but
    \begin{equation}\label{eq-contradiction}
        \mathcal{E}_{\rho}(v_j) >\rho^{ \beta_{p_j, \hat c}}.
    \end{equation}
    Let $h_j$ be harmonic replacements of $v_j$ and let $w_j=(v_j-h_j)/\varepsilon_j$. By applying Lemma~\ref{lem-compactness}, 
    \begin{equation*}
        \xi_j \to 0, \quad h_j \rightharpoonup h \ \text{in $H^1(B_1)$}, \quad v_j \to h \ \text{in $C^1(\overline B_R)$}, \quad w_j \rightharpoonup w \ \text{in $H^1(B_R)$},
    \end{equation*}
    for every $R \in (0,1)$. Moreover, $w \in W_0^{1,q}(B_1)$ satisfies
      \begin{equation*}
        -\Delta w=\mathrm{div}(Dh\log|Dh|) \quad \text{in $B_1$}.
    \end{equation*}
    
    We recall from \eqref{eq-normalized2} and \eqref{eq-harmonic-replacement} that
    \begin{equation}\label{eq-energy}
        1=(\mathcal{E}_{1}(v_j))^2=(\mathcal{E}_{1}(h_j))^2+\varepsilon_j^2\fint_{B_1}|Dw_j|^2.
    \end{equation}
    Since $h_j \to h$ weakly in $H^1(B_1)$, we use the weak lower semicontinuity to obtain that
    \begin{equation*}
        \mathcal{E}_1(h) \leq \liminf_{j \to \infty}\mathcal{E}_{1}(h_j) \leq 1.
    \end{equation*}
     We then pass to the limit in \eqref{eq-contradiction} and apply Lemma~\ref{lem-harmonic-rigidity} to find
    \begin{equation*}
        \rho \leq  \mathcal{E}_{\rho}(h) \leq \rho \mathcal{E}_1(h) \leq \rho.
    \end{equation*}
    The equality condition in Lemma~\ref{lem-harmonic-rigidity} shows that 
     \begin{equation*}
        h=P_H(x) = x^{\top}Hx \quad \text{for $H \in \mathrm{Sym}_d$, $\mathrm{tr}H=0$, and $\mathcal{E}_{1}(P_H)=1$}.
    \end{equation*}
    In particular, the limit function $w=w_H$ satisfies
    \begin{equation*}
        -\Delta w_H=\mathrm{div}(DP_H\log|DP_H|) \quad \text{in $B_1$}.
    \end{equation*}

    We now use the relation $v_j=h_j+\varepsilon_jw_j$ for each $j$ as follows.
    \begin{equation*}
        \begin{aligned}
             (\mathcal{E}_{\rho}(v_j))^2&=(\mathcal{E}_{\rho}(h_j))^2+2\varepsilon_j\mathcal{B}_{\rho}(h_j, w_j)+\varepsilon_j^2(\mathcal{E}_{\rho}(w_j))^2\\
             &\leq \rho^2+2\varepsilon_j\mathcal{B}_{\rho}(h_j, w_j)+O(\varepsilon_j^2),
        \end{aligned}
    \end{equation*}
    where we used \eqref{eq-energy}, Lemmas~\ref{lem-compactness} and \ref{lem-harmonic-rigidity}. Since $h_j \to h$ strongly in $H^1(B_{\rho})$ and $w_j \rightharpoonup w_H$ weakly in $H^1(B_{\rho})$, we have
    \begin{equation*}
        \mathcal{B}_{\rho}(h_j, w_j) \to \mathcal{B}_{\rho}(P_H, w_H).
    \end{equation*}
    By using Lemma~\ref{lem-identity}, Proposition~\ref{prop:mu-lb}, and $\log \rho<0$, we conclude that
    \begin{equation*}
        \begin{aligned}
            (\mathcal{E}_{\rho}(v_j))^2&\leq \rho^2+2\varepsilon_j \mu(H)\rho^2\log\rho+o(\varepsilon_j)\\
            &\leq \rho^2[1+|\log\rho|\varepsilon_j]+o(\varepsilon_j).
        \end{aligned}
    \end{equation*}
    On the other hand, Taylor's theorem together with the definition of $\beta_{p, \hat c}$ yields that
    \begin{equation*}
        \rho^{2\beta_{p_j, \hat c}}=\rho^2[1+2(1-\beta_{p_j, \hat c})|\log\rho|]+o(\varepsilon_j).
    \end{equation*}
    Since the condition $\hat c<1/2$ guarantees
    \begin{equation*}
        1-\beta_{p_j, \hat c}=(1-\hat c)\varepsilon_j+O(\varepsilon_j^2) >\varepsilon_j/2,
    \end{equation*}
    we have $(\mathcal{E}_{\rho}(v_j))^2 <\rho^{2\beta_{p_j, \hat c}}$ for all sufficiently large $j$. This contradicts the initial condition \eqref{eq-contradiction} and proves the lemma.
\end{proof}

In the remainder of this section, we fix $\rho=1/8$ and then set the constant $\varepsilon_{d,c}=\varepsilon_{d, 1/8, c}$ chosen in Lemma~\ref{lem-initial-step}. In what follows, we decrease $\varepsilon_{d,c}$, if necessary, so that
Lemma~\ref{lem-xi-bound} is applicable for every $p\in(2,2+\varepsilon_{d,c})$.

\begin{lemma}[Growth estimate at critical points]\label{lem-growth-critical}
    For $p \in (2, 2+\varepsilon_{d,c})$ with $c \in (0, 1/2)$, every $p$-harmonic function $u$ in $B_1$ with $Du(0)=0$ satisfies
    \begin{equation*}
        |Du(x)| \leq C\|u\|_{L^{\infty}(B_1)}|x|^{\beta_{p, \hat c}} \quad \text{for $x \in B_{1/4}$},
    \end{equation*}
    and so
    \begin{equation*}
        |u(x)-u(0)| \leq C\|u\|_{L^{\infty}(B_1)}|x|^{1+\beta_{p, \hat c}} \quad \text{for $x \in B_{1/4}$},
    \end{equation*}
    where $C>0$ depends only on $d$, $p$, and $c$.    
\end{lemma}

\begin{proof}
    Set $r_0=1/2$ and $r_k=r_0\rho^k$ for $k \in \mathbb{N}$. We first prove the excess decay 
    \begin{equation}\label{eq-excess-induction}
        \mathcal{E}_{r_k}(u) \leq \rho^{k\beta_{p, \hat c}}\mathcal{E}_{r_0}(u) \quad \text{for every $k \geq 0$}.
    \end{equation}
    If $\mathcal{E}_{r_k}(u)=0$, then $Du$ is constant in $B_{r_k}$ and so vanishes in $B_{r_k}$ due to the condition $Du(0)=0$. Thus, the excess vanishes at all smaller scales. 
    
    We next suppose that $\mathcal{E}_{r_k}(u)>0$ and set 
    \begin{equation*}
        a_k  \coloneqq (Du)_{B_{r_k}}, \quad v_k(x) \coloneqq \frac{u(r_kx)-u(0)-a_k \cdot (r_kx)}{r_k\mathcal{E}_{r_k}(u)}, \quad \text{and} \quad \xi_k \coloneqq \frac{a_k}{\mathcal{E}_{r_k}(u)}.
    \end{equation*}
    Then $(v_k, \xi_k)$ satisfies \eqref{eq-normalized1} and \eqref{eq-normalized2}. It follows from a direct calculation that 
    \begin{equation*}
        \begin{aligned}
        (Dv_k)_{B_{\rho}}&=\fint_{B_{\rho}}Dv_k(x)\,\mathrm{d}x=\frac{1}{\mathcal{E}_{r_k}(u)}\left(\fint_{B_{\rho}}Du(r_kx)\,\mathrm{d}x-a_k\right)\\
        &=\frac{1}{\mathcal{E}_{r_k}(u)}\left(\fint_{B_{r_{k+1}}}Du(y)\,\mathrm{d}y-a_k\right)=\frac{a_{k+1}-a_k}{\mathcal{E}_{r_k}(u)}
        \end{aligned}
    \end{equation*}
    and so
    \begin{equation*}
        \begin{aligned}
            (\mathcal{E}_{\rho}(v_k))^2=\fint_{B_{\rho}}|Dv_k-(Dv_k)_{B_{\rho}}|^2=\left(\frac{\mathcal{E}_{r_{k+1}}(u)}{\mathcal{E}_{r_k}(u)}\right)^2.
        \end{aligned}
    \end{equation*}
    Thus, we apply Lemma~\ref{lem-initial-step} to find that
    \begin{equation*}
         \frac{\mathcal{E}_{r_{k+1}}(u)}{\mathcal{E}_{r_k}(u)}=\mathcal{E}_{\rho}(v_k) \leq \rho^{\beta_{p, \hat c}},
    \end{equation*}
    which proves \eqref{eq-excess-induction} by induction.\\

    We next derive a pointwise estimate from the excess decay. An application of Lemma~\ref{lem-xi-bound} for $(v_k, \xi_k)$ gives that
    \begin{equation*}
        |a_k| \leq C_d\mathcal{E}_{r_k}(u) \quad \text{and} \quad \|v_k\|_{C^{1, \gamma_0}(\overline B_{2/3})} \leq C_d.
    \end{equation*}
    By scaling back and using the excess decay, we observe that
    \begin{equation*}
        \sup_{B_{2r_k/3}}|Du| \leq 2C_d\mathcal{E}_{r_k}(u) \leq 2C_d \rho^{k\beta_{p, \hat c}}\mathcal{E}_{r_0}(u).
    \end{equation*}
    For each $0<|x|<1/4$, choose $k$ so that 
    \begin{equation*}
        \frac{2r_0}{3}\rho^{k+1}=\frac{2r_{k+1}}{3} <|x| \leq \frac{2r_k}{3}=\frac{2r_0}{3}\rho^k.
    \end{equation*}
    Then we have
    \begin{equation*}
        |Du(x)|\leq C\mathcal{E}_{r_0}(u) |x|^{\beta_{p, \hat c}}.
    \end{equation*}
    On the other hand, we obtain
    \begin{equation*}
        \mathcal{E}_{r_0}(u) \leq \left(\fint_{B_{r_0}}|Du|^2\right)^{1/2} \leq \left(\fint_{B_{r_0}}|Du|^p\right)^{1/p} \leq C\|u\|_{L^{\infty}(B_1)},
    \end{equation*}
    where we used H\"older's inequality and the Caccioppoli inequality for $p$-harmonic functions. A combination of these estimates leads to the desired conclusion.
\end{proof}

We now extend the estimate from critical points to arbitrary points. The proof separates two regimes depending on whether the size of the gradient $|Du(0)|$ is small or not. The strict inequality $\beta_{p, \hat c}>\beta_{p,c}$ allows the first regime to be treated by compactness. 

\begin{lemma}[Near-critical regime]\label{lem-near-critical}
    Let $p\in (2, 2+\varepsilon_{d,c})$ with $c \in (0, 1/2)$. Then there exist $\rho_{\ast}, \eta_{\ast} \in (0, 1/8)$, depending only on $d$, $p$, and $c$, such that  if $u$ is $p$-harmonic in $B_1$ and 
    \begin{equation*}
        u(0)=0, \quad \|u\|_{L^{\infty}(B_1)} \leq 1, \quad |Du(0)| \leq \eta_{\ast},
    \end{equation*}
    then 
    \begin{equation*}
        \sup_{B_{\rho_{\ast}}}|u| \leq \rho_{\ast}^{1+\beta_{p,c}}.
    \end{equation*}
\end{lemma}

\begin{proof}
    By Lemma~\ref{lem-growth-critical}, every $p$-harmonic function $U$ in $B_1$ satisfying 
    \begin{equation*}
         U(0)=0, \quad \|U\|_{L^{\infty}(B_1)} \leq 1, \quad DU(0)=0
    \end{equation*}
    enjoys
    \begin{equation*}
        \sup_{B_r}|U| \leq C_0r^{1+\beta_{p, \hat c}} \quad \text{for $r \in (0, 1/4]$},
    \end{equation*}
    where $C_0=C_0(d, p, c)>0$. We choose $\rho_{\ast} \in (0, 1/8)$ so small that 
    \begin{equation*}
        C_0 \rho_{\ast}^{\beta_{p, \hat c}-\beta_{p, c}} \leq \frac{1}{4}.
    \end{equation*}

    We prove by contradiction: suppose that there exists no $\eta_{\ast}>0$ with the desired property. Then there exists a sequence of $p$-harmonic functions $\{u_j\}_{j=1}^{\infty}$ such that
    \begin{equation*}
         u_j(0)=0, \quad \|u_j\|_{L^{\infty}(B_1)} \leq 1, \quad |Du_j(0)| \leq 1/j,
    \end{equation*}
    but
    \begin{equation*}
        \sup_{B_{\rho_{\ast}}}|u_j| > \rho_{\ast}^{1+\beta_{p, c}}.
    \end{equation*}
    The interior $C^{1, \gamma_0}$ estimate for $p$-harmonic functions (see \cite{Ura68} for instance) and the compact embedding guarantee, up to a subsequence, 
    \begin{equation*}
        u_j \to U \quad \text{in $C^1_{\mathrm{loc}}(B_1)$},
    \end{equation*}
    where $U$ is $p$-harmonic, $U(0)=0$, $DU(0)=0$, and $\|U\|_{L^{\infty}(B_1)}\leq 1$. In particular, $U$ satisfies 
     \begin{equation*}
        \sup_{B_{\rho_{\ast}}}|U| \leq C_0\rho_{\ast}^{1+\beta_{p, \hat c}}\leq \frac{1}{4}\rho_{\ast}^{1+\beta_{p, c}},
    \end{equation*}
    which contradicts the preceding inequality for $u_j$ with sufficiently large $j$.
\end{proof}

\begin{lemma}[Non-degenerate regime]\label{lem-non-degenerate}
    Let $p>2$ and $M>0$. Then there exist $\sigma \in (0, 1/8)$ and $C>0$, depending only on $d$, $p$, and $M$, such that  if $u$ is $p$-harmonic in $B_2$ and 
    \begin{equation*}
        \|u\|_{L^{\infty}(B_2)} \leq M, \quad |Du(0)|=1,
    \end{equation*}
    then $|Du| \geq 1/2$ in $B_{2\sigma}$ and 
    \begin{equation*}
        \sup_{B_r}|u-u(0)-Du(0)\cdot x| \leq Cr^{2} \quad \text{for any $0<r \leq \sigma$}.
    \end{equation*}
\end{lemma}

\begin{proof}
    By the interior $C^{1, \gamma_0}$ estimate for $p$-harmonic functions, there exist $\gamma_0 \in (0,1)$ and $C_0>0$, depending only on $d$ and $p$, such that
    \begin{equation*}
        \|u\|_{C^{1, \gamma_0}(\overline B_1)} \leq C_0M.
    \end{equation*}
    We choose $\sigma \in (0, 1/8)$ so small that
    \begin{equation*}
        C_0M(2\sigma)^{\gamma_0} \leq \frac{1}{2}.
    \end{equation*}
    Since $|Du(0)|=1$, it follows that $|Du| \geq 1/2$ in $B_{2\sigma}$. 
    
    On the other hand, we recall from \cite[Corollary~2.6]{JLM01} and \cite[Theorem~6]{KMP12} that the $p$-harmonic function $u$ can be understood as a viscosity solution of the normalized $p$-Laplace equation
    \begin{equation*}
        a_{ij}D_{ij}u=0 \quad \text{for $a_{ij} \coloneqq \delta_{ij}+(p-2)\frac{D_iuD_ju}{|Du|^2}$}.
    \end{equation*}
    Since $I_d\leq A=(a_{ij}) \leq (p-1)I_d $, the equation is uniformly elliptic. Moreover, the $C^{\gamma_0}$ regularity of $Du$ together with the lower bound $|Du| \geq 1/2$ allows us to employ the interior Schauder estimate to conclude that
    \begin{equation*}
        \|u\|_{C^{2, \gamma_0}(\overline B_{\sigma})} \leq C(d, p, M).\qedhere
    \end{equation*}
\end{proof}

\begin{lemma}[Growth estimate at general points]\label{lem-iteration}
    Let $p \in (2, 2+\varepsilon_{d, c})$ with $c \in (0, 1/2)$. Suppose that $u$ is $p$-harmonic in $B_1$ and 
    \begin{equation*}
        u(0)=0, \quad \|u\|_{L^{\infty}(B_1)} \leq 1.
    \end{equation*}
    Then 
    \begin{equation*}
        \sup_{B_r}|u-Du(0) \cdot x| \leq Cr^{1+\beta_{p, c}} \quad \text{for any $0<r \leq 1/4$},
    \end{equation*}
    where $C>0$ depends only on $d$, $p$, and $c$.
\end{lemma}

\begin{proof}
    Let $\rho_{\ast}, \eta_{\ast} \in (0, 1/8)$ be the constants chosen in Lemma~\ref{lem-near-critical}. We decrease $\eta_{\ast}$, if necessary, so that $\rho_{\ast} \eta_{\ast}^{-1/\beta_{p, c}}>4$. We set $r_k \coloneqq \rho_{\ast}^k$ and consider a rescaled function
    \begin{equation*}
        u_k(x) \coloneqq \frac{u(r_kx)}{r_k^{1+\beta_{p, c}}}.
    \end{equation*}
    Then $u_k(0)=0$ and $|Du_k(0)|=r_k^{-\beta_{p, c}}|Du(0)|$.

\medskip
\noindent\textit{(Case 1: $|Du(0)| \leq \eta_{\ast}$.)} We observe that whenever 
    \begin{equation*}
        \sup_{B_{r_k}}|u| \leq r_k^{1+\beta_{p, c}} \quad \text{and} \quad |Du(0)| \leq \eta_{\ast} r_k^{\beta_{p, c}},
    \end{equation*}
    we can apply Lemma~\ref{lem-near-critical} for $u_k$ to obtain
    \begin{equation*}
        \sup_{B_{r_{k+1}}}|u| \leq r_{k+1}^{1+\beta_{p, c}}.
    \end{equation*}
    If $Du(0)=0$, then this argument can be iterated for every $k$, and the desired growth follows.

    We next suppose that $0<|Du(0)|\leq \eta_{\ast}$, and let $N \geq 1$ be the first integer such that $|Du(0)|>\eta_{\ast} r_N^{\beta_{p, c}}$. The preceding iteration argument is valid up to scale $r_N$:
    \begin{equation*}
        \sup_{B_{r_N}}|u| \leq r_N^{1+\beta_{p, c}}.
    \end{equation*}
    If we set the intermediate scale $\lambda \coloneqq |Du(0)|^{1/\beta_{p, c}}$, then the minimality of $N$ shows 
    \begin{equation}\label{eq-minimal-N}
        4\lambda<\rho_{\ast} \eta_{\ast}^{-1/\beta_{p, c}}\lambda \leq r_N <\eta_{\ast}^{-1/\beta_{p, c}}\lambda.
    \end{equation}
    Let us consider 
    \begin{equation*}
        u_{\lambda}(x) \coloneqq \frac{u(\lambda x)}{\lambda^{1+\beta_{p, c}}},
    \end{equation*}
    which is $p$-harmonic in $B_2$ and 
    \begin{equation*}
        |Du_{\lambda}(0)|=1, \quad \|u_{\lambda}\|_{L^{\infty}(B_2)} \leq \left(\frac{r_N}{\lambda}\right)^{1+\beta_{p, c}} \leq \eta_{\ast}^{-(1+\beta_{p, c})/\beta_{p, c}} \eqqcolon M_{d, p, c}.
    \end{equation*}
    For $0<r \leq \sigma \lambda$, an application of Lemma~\ref{lem-non-degenerate} yields that
    \begin{equation*}
        \sup_{B_r}|u-Du(0) \cdot x| \leq C\lambda^{1+\beta_{p, c}}\left(\frac{r}{\lambda}\right)^2\leq Cr^{1+\beta_{p, c}},
    \end{equation*}
    where we used the condition $\beta_{p, c}<1$ in the last inequality. 

    For $\sigma \lambda < r \leq r_N$, it follows from \eqref{eq-minimal-N} that
    \begin{equation*}
        \sup_{B_r}|u-Du(0)\cdot x| \leq r_N^{1+\beta_{p, c}}+|Du(0)|r \leq C\lambda^{1+\beta_{p, c}}+\lambda^{\beta_{p, c}}r \leq Cr^{1+\beta_{p, c}}.
    \end{equation*}
    
    Finally, if $r_N < r \leq 1/4$, choose $k<N$ so that $r_{k+1}<r \leq r_k$. Since
    \begin{equation*}
        \sup_{B_r}|u| \leq r_k^{1+\beta_{p, c}} \leq \rho_{\ast}^{-(1+\beta_{p, c})}r^{1+\beta_{p, c}} \quad \text{and} \quad |Du(0)| \leq \eta_{\ast} r_{N-1}^{\beta_{p, c}}\leq \eta_{\ast} \rho_{\ast}^{-\beta_{p, c}}r^{\beta_{p, c}},
    \end{equation*}
    the growth estimate also holds in this range.

\medskip
\noindent\textit{(Case 2: $|Du(0)| > \eta_{\ast}$.)} The interior $C^{1, \gamma_0}$ estimate for $u$ gives 
    \begin{equation*}
        \|u\|_{C^{1, \gamma_0}(B_{1/2})} \leq C_0=C_0(d, p).
    \end{equation*}
    Choose $r_{\ast}=r_{\ast}(d, p, c) \in (0, 1/8)$ small enough so that $C_0(2r_{\ast})^{\gamma_0} \leq \eta_{\ast}/2$ and that $|Du| \geq \eta_{\ast}/2$ in $B_{2r_{\ast}}$. By repeating the argument in Lemma~\ref{lem-non-degenerate}, we have
    \begin{equation*}
        \|u\|_{C^{2, \gamma_0}(B_{r_{\ast}})} \leq C(d, p, c),
    \end{equation*}
    which implies the growth estimate for $0<r \leq r_{\ast}$. For $r_{\ast} < r \leq 1/4$, the growth estimate immediately follows from $\|u\|_{L^{\infty}(B_1)} \leq 1$ and the interior gradient bound. This completes the proof. 
\end{proof}

We are now ready to prove our main theorems.
\begin{proof}[Proof of Theorem~\ref{thm-main-p-harmonic}]
    Fix $x_0 \in B_{1/2}$ and set $R=1/2$. If $\|u\|_{L^{\infty}(B_1)}=0$, there is nothing to prove. Otherwise, we define
    \begin{equation*}
        U(y) \coloneqq \frac{u(x_0+Ry)-u(x_0)}{2\|u\|_{L^{\infty}(B_1)}}.
    \end{equation*}
    Then $U$ is $p$-harmonic in $B_1$, $U(0)=0$, and $\|U\|_{L^{\infty}(B_1)} \leq 1$. By applying Lemma~\ref{lem-iteration} and scaling back, we obtain
    \begin{equation*}
        |u(x)-u(x_0)-Du(x_0) \cdot (x-x_0)| \leq C\|u\|_{L^{\infty}(B_1)} |x-x_0|^{1+\beta_{p, c}}.
    \end{equation*}
    A standard covering lemma finishes the proof of Theorem~\ref{thm-main-p-harmonic}.
\end{proof}

\begin{proof}[Proof of Theorem~\ref{cor-conjecture}]
For $c=1/4$, we set $\varepsilon_d=\varepsilon_{d,1/4}$, and let $C=C(d, p, 1/4)>0$ be the constant in Theorem~\ref{thm-main-p-harmonic}. For every $p\in(2,2+\varepsilon_d)$, Theorem~\ref{thm-main-p-harmonic} verifies the hypotheses of Theorem~\ref{thm-ATU18} with
\begin{equation*}
    \alpha_0=\beta_{p,1/4}
    =\frac1{p-1}+\frac{p-2}{4}
    \in\left(\frac1{p-1},1\right)
\end{equation*}
and $C_0=C$. Thus, Theorem~\ref{thm-ATU18} yields the asserted local $C^{p'}$ regularity.
\end{proof}

\begin{appendix}
\section{Sharp lower bound for \texorpdfstring{$\mu$}{mu}}\label{sec-appendix}

\begin{lemma}
\label{prop:wvar}
Let $m \geq 2$. Let $\nu=(\nu_1,\ldots,\nu_m)\in\mathbb R^m$ satisfy $\nu_i\geq0$ for every $i$ and
$\sum_{i=1}^m\nu_i=1$.
Then, for every $\lambda=(\lambda_1,\ldots,\lambda_m)\in\mathbb R^m$,
\begin{equation}\label{eq:wvar}
    \begin{aligned}
        V_\nu(\lambda)
        \coloneqq
        \sum_{i=1}^m\nu_i \lambda_i^2
        -\left(\sum_{i=1}^m\nu_i \lambda_i\right)^2
        \leq
        \left(\frac14+\frac12\sum_{i=1}^m\nu_i^2\right)
        |\lambda|^2.
    \end{aligned}
\end{equation}
For $\lambda\neq0$, equality holds if and only if, up to a
permutation of the coordinates,
\begin{equation}\label{eq:wvar-eq}
    \nu=\left(\frac12,\frac12,0,\ldots,0\right)
    \quad\text{and}\quad
    \lambda=(a,-a,0,\ldots,0)
\end{equation}
for some $a\neq0$.
\end{lemma}

\begin{proof}
The case $\lambda=0$ is immediate. Since both sides of \eqref{eq:wvar} are homogeneous of degree two in $\lambda$, it suffices to assume $|\lambda|=1$ and prove
\begin{equation}\label{eq:wvar-red}
    V_\nu(\lambda)
    \leq\frac14+\frac12\sum_{i=1}^m\nu_i^2.
\end{equation}

Let $\Lambda=(\lambda_1^2,\ldots,\lambda_m^2)$ and $s_k=\sum_{i=1}^m \lambda_i^k$. We consider the positive definite matrix $A\coloneqq I_m+2\lambda \lambda^\top$, whose inverse is $A^{-1}=I_m-\frac23\lambda \lambda^\top$. Then a direct computation gives
\begin{equation}\label{eq:wvar-csq}
    \begin{aligned}
        V_\nu(\lambda)-\frac12\sum_{i=1}^m\nu_i^2
        &=-\frac12\nu^\top A\nu+\Lambda\cdot\nu\\
        &=\frac12\Lambda^\top A^{-1}\Lambda
        -\frac12(\nu-A^{-1}\Lambda)^\top
        A(\nu-A^{-1}\Lambda)\\
        &\leq\frac12\Lambda^\top A^{-1}\Lambda
        =\frac12\left(s_4-\frac23s_3^2\right).
    \end{aligned}
\end{equation}

It remains to show that $s_4-2s_3^2/3\leq 1/2$.
Using $s_2=1$, we obtain
\begin{equation}\label{eq:wvar-pair-bd}
    \begin{aligned}
        s_4-s_3^2
        &=\sum_{i<j}\lambda_i^2\lambda_j^2(\lambda_i-\lambda_j)^2\\
        &\leq2\sum_{i<j}\lambda_i^2\lambda_j^2(\lambda_i^2+\lambda_j^2)\\
        &\leq2\sum_{i<j}\lambda_i^2\lambda_j^2
        =1-s_4.
    \end{aligned}
\end{equation}
Hence,
\begin{equation}\label{eq:wvar-mom}
    s_4-\frac23s_3^2
    =\frac12(2s_4-s_3^2)-\frac16s_3^2
    \leq\frac12.
\end{equation}
Together with \eqref{eq:wvar-csq}, this proves \eqref{eq:wvar-red}.

We next determine when equality holds. Suppose that equality holds in \eqref{eq:wvar-red}. Then every inequality above must be an equality.
In particular, \eqref{eq:wvar-mom} gives $s_3=0$. Since $|\lambda|=1$ and $s_3=0$, there are at least two nonzero components. For any $i<j$ with $\lambda_i\lambda_j\neq0$, equality in \eqref{eq:wvar-pair-bd} implies $\lambda_i=-\lambda_j$ and $\lambda_i^2+\lambda_j^2=1$. Thus, all other components vanish and after a permutation, we have $\lambda=(1,-1,0,\ldots,0)/\sqrt2$. Moreover, \eqref{eq:wvar-csq} gives $\nu=A^{-1}\Lambda$. Since $\lambda^\top \Lambda=s_3=0$, we obtain $\nu=\Lambda=(1/2,1/2,0,\ldots,0)$.

For arbitrary $\lambda\neq0$, applying this conclusion to $\lambda/|\lambda|$ yields \eqref{eq:wvar-eq}. Conversely, direct substitution shows that \eqref{eq:wvar-eq} gives equality in \eqref{eq:wvar}. This completes the proof.
\end{proof}

\begin{proposition}\label{prop:mu-lb}
Let $d\geq2$, and let
$H\in\mathrm{Sym}_d\setminus\{0\}$ satisfy
$\operatorname{tr}H=0$. Then
\begin{equation*}
    \mu(H)\geq-\frac12.
\end{equation*}
Equality holds if and only if, after an orthogonal change
of coordinates,
\begin{equation*}
    H=\operatorname{diag}(a,-a,0,\ldots,0) \quad \text{for some $a\neq0$}.
\end{equation*}
\end{proposition}

\begin{proof}
After an orthogonal change of coordinates, write
\begin{equation*}
    H=\operatorname{diag}(\lambda_1,\ldots,\lambda_m,0,\ldots,0),
\end{equation*}
where $m=\operatorname{rank}(H)$ and $\lambda_i\neq0$ for $i=1,\ldots,m$.
Let $\lambda=(\lambda_1,\ldots,\lambda_m)$. Since $H\neq0$ and $\operatorname{tr}H=0$, we have $m\geq2$ and $\sum_{i=1}^m\lambda_i=0$.
Set
\begin{equation*}
    \mathcal I(H)\coloneqq
    \fint_{\mathbb S^{d-1}}
    \frac{P_{H^4}P_{H^2}-P_{H^3}^2}{P_{H^2}^2}.
\end{equation*}
By Lemma~\ref{lem-mu},
$\mu(H)=-2\mathcal I(H)/|\lambda|^2$. Thus, it suffices to show that
$\mathcal I(H)\leq|\lambda|^2/4$.

We first express $\mathcal I(H)$ as a Gaussian integral.
Let $G=(G_1,\ldots,G_d)$ be a standard Gaussian random
vector in $\mathbb R^d$, with density
$(2\pi)^{-d/2}e^{-|g|^2/2}$ for $g\in \mathbb{R}^d$.
The direction $G/|G|$ is uniformly distributed on
$\mathbb S^{d-1}$, and the integrand defining
$\mathcal I(H)$ is homogeneous of degree zero.
Therefore,
\begin{equation}\label{eq:gauss-rep}
    \mathcal I(H)
    =\mathbb E\left[
        \sum_{i=1}^m\nu_i\lambda_i^2
        -\left(\sum_{i=1}^m\nu_i\lambda_i\right)^2
    \right],
\end{equation}
where
\begin{equation*}
    \nu_i\coloneqq
    \frac{\lambda_i^2G_i^2}
         {\sum_{j=1}^m\lambda_j^2G_j^2}.
\end{equation*}
These weights are well defined almost surely and satisfy
$\sum_{i=1}^m\nu_i=1$.
Expanding the square in \eqref{eq:gauss-rep}, we obtain
\begin{equation*}
    \mathcal I(H)
    =\sum_{i<j}(\lambda_i-\lambda_j)^2
    \mathbb E[\nu_i\nu_j].
\end{equation*}

We next compute $\mathbb E[\nu_i\nu_j]$ for $i\neq j$.
Let $\Gamma_i=G_i^2/2$ and
$S=\sum_{k=1}^m\lambda_k^2\Gamma_k$.
The variables $\Gamma_i$ are independent Gamma random variables
with shape $1/2$ and rate $1$.
For $t\geq0$, define
\begin{equation*}
    \Lambda_i(t)\coloneqq\frac{\lambda_i^2}{1+\lambda_i^2t},
    \quad
    \mathcal{L}(t)\coloneqq\prod_{i=1}^m(1+\lambda_i^2t)^{-1/2}.
\end{equation*}
Using $S^{-2}=\int_0^\infty te^{-tS}\,\mathrm{d}t$
and Tonelli's theorem, we have
\begin{equation*}
    \begin{aligned}
        \mathbb E[\nu_i\nu_j]
        &=\lambda_i^2\lambda_j^2
        \mathbb E\left[\frac{\Gamma_i\Gamma_j}{S^2}\right]\\
        &=\lambda_i^2\lambda_j^2
        \int_0^\infty
        t\,\mathbb E[\Gamma_i\Gamma_je^{-tS}]\,\mathrm{d}t.
    \end{aligned}
\end{equation*}
On the other hand, Gaussian integration gives
\begin{equation*}
    \mathbb E[e^{-t\lambda_i^2\Gamma_i}]
    =\frac1{\sqrt{2\pi}}
    \int_{\mathbb R}e^{-(1+\lambda_i^2t)g^2/2}\,\mathrm{d}g
    =(1+\lambda_i^2t)^{-1/2}.
\end{equation*}
Differentiating with respect to $t$, we obtain
\begin{equation*}
    \lambda_i^2\mathbb E[\Gamma_ie^{-t\lambda_i^2\Gamma_i}]
    =\frac12 \Lambda_i(t)(1+\lambda_i^2t)^{-1/2}.
\end{equation*}
Since $\Gamma_1,\ldots,\Gamma_m$ are independent, we have
\begin{equation*}
    \mathbb E[\nu_i\nu_j]
    =\frac14\int_0^\infty
    t\mathcal{L}(t)\Lambda_i(t)\Lambda_j(t)\,\mathrm{d}t,
\end{equation*}
and so we find
\begin{equation}\label{eq:i-int}
    \mathcal I(H)
    =\frac14\int_0^\infty
    t\mathcal{L}(t)\sum_{i<j}
    \Lambda_i(t)\Lambda_j(t)(\lambda_i-\lambda_j)^2\,\mathrm{d}t.
\end{equation}

To estimate the sum in \eqref{eq:i-int}, set
\begin{equation*}
    M_1(t)\coloneqq\sum_{i=1}^m \Lambda_i(t),
    \quad
    M_2(t)\coloneqq\sum_{i=1}^m \Lambda_i(t)^2,
\end{equation*}
and let $\tilde\nu_i(t)=\Lambda_i(t)/M_1(t)$. Then $\tilde \nu_i(t)>0$, $\sum_{i=1}^m\tilde \nu_i(t)=1$, and
\begin{equation}\label{eq-V}
    \sum_{i<j}\Lambda_i(t)\Lambda_j(t)(\lambda_i-\lambda_j)^2
    =M_1(t)^2V_{\tilde \nu(t)}(\lambda).
\end{equation}
Applying Lemma~\ref{prop:wvar} with $\nu=\tilde \nu(t)$, we obtain
\begin{equation}\label{eq:i-sum-ub}
    \sum_{i<j}\Lambda_i(t)\Lambda_j(t)(\lambda_i-\lambda_j)^2
    \leq
    \left(\frac14M_1(t)^2+\frac12M_2(t)\right)|\lambda|^2.
\end{equation}
Since $\log \mathcal{L}(t)=-\frac12\sum_{i=1}^m\log(1+\lambda_i^2t)$,
we have
\begin{equation*}
    \mathcal{L}'(t)=-\frac12M_1(t)\mathcal{L}(t),
    \qquad
    M_1'(t)=-M_2(t).
\end{equation*}
Thus,
\begin{equation}\label{eq:d2}
    \mathcal{L}''(t)=\mathcal{L}(t)\left(\frac14M_1(t)^2+\frac12M_2(t)\right).
\end{equation}
Substituting \eqref{eq:i-sum-ub} and \eqref{eq:d2}
into \eqref{eq:i-int} gives
\begin{equation*}
    \mathcal I(H)
    \leq\frac14|\lambda|^2
    \int_0^\infty t\mathcal{L}''(t)\,\mathrm{d}t.
\end{equation*}

It remains to evaluate this integral. Since each $\lambda_i$ is nonzero,
\begin{equation*}
    \mathcal{L}(t)=O(t^{-m/2}),
    \qquad
    \mathcal{L}'(t)=O(t^{-m/2-1})
    \qquad\text{as }t\to\infty.
\end{equation*}
Integrating by parts on $[0,R]$ and letting $R\to\infty$, we obtain
\begin{equation}\label{eq:d2-int}
    \int_0^\infty t\mathcal{L}''(t)\,\mathrm{d}t
    =\lim_{R\to\infty}\bigl(R\mathcal{L}'(R)-\mathcal{L}(R)+\mathcal{L}(0)\bigr)
    =1.
\end{equation}
Therefore, $\mathcal I(H)\leq|\lambda|^2/4$, which proves $\mu(H)\geq-1/2$.

We next determine when equality holds. Combining \eqref{eq:i-int}, \eqref{eq-V}, \eqref{eq:d2}, and \eqref{eq:d2-int}, we have
\begin{equation}\label{eq:mu-def}
    \frac14|\lambda|^2-\mathcal I(H)
    =
    \frac14\int_0^\infty
    t\mathcal{L}(t)M_1(t)^2
    \left[
        \left(\frac14+\frac12\sum_{i=1}^m\tilde \nu_i(t)^2\right)
        |\lambda|^2-V_{\tilde \nu(t)}(\lambda)
    \right]\,\mathrm{d}t.
\end{equation}
The bracketed term is nonnegative by Lemma~\ref{prop:wvar}.
If $\mu(H)=-1/2$, the left-hand side vanishes. Since the bracketed term is continuous and $t\mathcal{L}(t)M_1(t)^2>0$ for $t>0$, equality holds in Lemma~\ref{prop:wvar} with $\nu=\tilde \nu(t)$ for every $t>0$. All components of $\tilde \nu(t)$ are positive, so the equality condition in Lemma~\ref{prop:wvar} gives $m=2$ and $\lambda_1=-\lambda_2$. Thus, after an orthogonal change of coordinates, $H=\operatorname{diag}(a,-a,0,\ldots,0)$ for some $a\neq0$.

Conversely, suppose that $H$ has this form. Then $\lambda_1^2=\lambda_2^2$, and hence $\tilde \nu_1(t)=\tilde \nu_2(t)=1/2$ for every $t>0$. By the equality condition in Lemma~\ref{prop:wvar}, the bracketed term in \eqref{eq:mu-def} vanishes. It follows that $\mathcal I(H)=|\lambda|^2/4$ and $\mu(H)=-1/2$. This completes the proof.
\end{proof}
\end{appendix}


\end{document}